\documentclass[11pt,a4paper]{article}

\usepackage{amsmath}
\usepackage{amsthm}
\usepackage{pb-diagram}
\usepackage{xypic}
\usepackage{amsfonts}
\usepackage{amsmath,arydshln,multirow}

\usepackage{pifont}
\usepackage{arydshln}
\usepackage{cases}
\usepackage{amsmath}
\usepackage{amsfonts}
\usepackage{bm}
\usepackage{arydshln}
\usepackage{amsfonts,amsmath,amssymb,amscd,bbm,amsthm,mathrsfs,dsfont}
\usepackage{mathrsfs}
\usepackage{pb-diagram}
\usepackage{amssymb}
\usepackage{xypic}
\usepackage[dvips]{graphicx}
\usepackage[all]{xy}
\usepackage{CJK}
\usepackage[dvipsnames]{xcolor}
\usepackage{mathtools}
\newtheorem{Theorem}{Theorem}[section]
\newtheorem{Lemma}[Theorem]{Lemma}

\newtheorem{Definition}[Theorem]{Definition}
\newtheorem{Corollary}[Theorem]{Corollary}
\newtheorem{Proposition}[Theorem]{Proposition}
\newtheorem{Example}[Theorem]{Example}
\newtheorem{Remark}[Theorem]{Remark}

\begin{document}
\title{ On topological representation theory from quivers\thanks{Project supported by the National Natural Science Foundation of China(No.12071422, No. 11971144), the Zhejiang Provincial Natural Science Foundation of China (No.LY19A010023)  and High-level Scientific Research Foundation of Hebei Province.}}
 \author{
 Fang Li\thanks{fangli@zju.edu.cn;\; Department of Mathematics, Zhejiang University,
 Hangzhou 310027, China },$\;\;\;\;$ Zhihao Wang\thanks{zhihaowang@zju.edu.cn;\; Department of Mathematics, Zhejiang University,
 Hangzhou 310027, China},$\;\;\;\;$ Jie Wu, \thanks{ wujie@hebtu.edu.cn;\; Center of Topology and Geometry Based Technology, Hebei Normal University, Shijiazhuang 050024, China; \; School of Mathematical Sciences, Hebei Normal University, Shijiazhuang 050024, China} $\;\;\;\;$ Bin Yu\thanks{binyu@cjlu.edu.cn;\; Department of Mathematics,
China Jiliang University, Hangzhou 310018, P.R.China}
}

\maketitle
\begin{abstract}
 In this work, we introduce {\em topological representations of a quiver} as a system consisting of topological spaces and its relationships determined by the quiver. Such a setting gives a natural connection between topological representations of a quiver and diagrams of topological spaces.

First,  we investigate the relation between the category of topological
representations and that of linear representations of a quiver via  $P(\Gamma)$-$\mathcal{TOP}^o$ and $k\Gamma$-Mod, concerning (positively) graded or vertex (positively) graded modules.
Second, we  discuss the homological theory of topological
representations of quivers via  $\Gamma$-limit $Lim^{\Gamma}$ and using it, define the homology groups of topological
representations of quivers via $H_n$. It is found that some properties
of a quiver can be read from homology groups.
Third,  we investigate  the homotopy theory of topological representations of quivers.
We define the homotopy equivalence between two morphisms in $\textbf{Top}\mathrm{-}\textbf{Rep}\Gamma$ and show that the parallel Homotopy Axiom also holds for top-representations based on the homotopy equivalence.
 Last, we mainly obtain the functor $At^{\Gamma}$ from $\textbf{Top}\mathrm{-}\textbf{Rep}\Gamma$ to
$\textbf{Top}$ and show that $At^{\Gamma}$ preserves homotopy equivalence between morphisms. The relationship is  established between the homotopy groups of a top-representation $(T,f)$ and the homotopy groups of $At^{\Gamma}(T,f)$.

\end{abstract}
\textbf{2010 Mathematics Subject
Classifications:}\;16G20, 16E30,  20M30; 46H15, 54H10, 55Q05
%\section*{Content}

\tableofcontents

\section{Introduction and preliminaries}

Algebraic representation theory is a very important theory for us to understand and describe the structure of an algebra. It is also an important method to realize an abstract algebraic structure with concrete examples. The quiver theory is crucial to this theory.
\\
$\bullet$ If the field, as a one-dimensional linear spaces over itself, placed on the vertices of a quiver are all the same with the base field,  we then  obtain the path algebra of this quiver.
\\
$\bullet$ If one puts  non-trivial linear spaces on the vertices of a quiver, we then obtain a representation of this quiver, which is equivalent to the representation of the path algebra of the quiver.
\\
$\bullet$ If one places some algebraic bimodules on the vertices of a quiver, then we  obtain the representation of a  generalized path algebra (in fact, a tensor algebra), see \cite{LL}.

 Furthermore, one may ask what would happen if we put objects with more delicate structures on the vertices of a quiver.

In an early work \cite{LC}, we considered an extreme case: replacing linear spaces at vertices with sets and linear maps with general maps, then the resulting structure is called the \textbf{set-representation} of the quiver $\Gamma$. It was proved that the category of set-representations of $\Gamma$ is equivalent to the $S$-system category of the multiplicative semigroup $P(\Gamma)$ of $\Gamma$. Its  biggest difference is to not be longer an abelian category.

Why is the fact that representation categories are abelian important for us? The reason is that we consider it from the perspective of  quivers (possibly, with relations). We hope to achieve the characterization of quivers or its path algebras (with corresponding relations) through the homological properties of the categories.

If we look at objects on vertices,  and look at quivers only as the relationships among these objects on vertices, we do not need to pay too much attention to that whether the {\em categories} are abelian categories.

From this point of view, we can consider to put some objects on vertices that we want to care about, such as topological spaces. Therefore, we lead to the idea of {\em topological representation theory} of quivers.

We give a metaphor. A standard barbell consists of barbell bar and barbell pieces. When barbell pieces are very small, the focus of the barbell is the bar. When barbell pieces are replaced bigger and bigger, the focus of the barbell will become to the barbell pieces, no longer to the barbell bar.

The focus of the linear representations of a quiver is the quiver itself, while it is hoped that the quiver can be described by its representation. When we consider topological representations of a quiver on its vertices where topological spaces are put, the focus of topological representations of a quiver is no longer to depict the quiver, but to reflect the relationship among topological spaces placed on the vertices of the quiver. In this way, we can take topological representations of a quiver as a system consisting of topological spaces and its relationships determined by the quiver. Such a setting gives a natural connection between topological representations of a quiver and diagrams of topological spaces, although the emphasized views may be slightly different, where we concern more on the topological features of the quiver from the views of representation theory.

Now we are replacing the way of using representation category by topological space category to characterize a quiver. We call this phenomenon {\em topologization}, which is similar to the statement of geometrization. In this sense, this work is entitled as {\em topological representation theory}.
\vspace{4mm}

To begin, let us state some basic definitions and notations that will be used in this paper.

Let $X$ be a topological space, and let $G$ be a group. Suppose we have a
map $\pi: G\times X\rightarrow X$ defined by sending $(g, x)$ to $gx$ for
any $g\in G$ and $x\in X$. This leads to the following definition.

\begin{Definition}
Let $X$ be a topological space $G$ be a group. For any $g\in G$,
the left action ${L_{g}}$ on $X$ is a topological continuous map
$$L_{g}: X\rightarrow X$$
given by $x\mapsto gx$.
\end{Definition}

\begin{Definition}
Let $S$ be a semigroup and $M$ a non-empty topological space. If the map
$L_S: S\times M\longrightarrow M$ satisfies $L_{s_{2}s_{1}}(m)=L_{s_{2}}\big(L_{s_{2}}(m)\big)$
%$\varphi(s_{2},\varphi(s_{1},m))=\varphi(s_{2}s_{1},m)$,
 $\forall
s_{1},s_{2}\in S$, $\forall m\in M$, then $(M, L_S)$ is called
a {\bf left $S$-topological system}, or says, $S$ {\bf acts on the left of} $M$.

For short, we denote $L_{s}(m)$ by $sm$, left $S$-topological system
$(M, L_S)$ just as $M$. Similarly, we can define {\bf right
$S$-topological systems}.
\end{Definition}

\begin{Definition}\label{maindef}
Let $M,N$ are two $S$-Systems of topological spaces, a continuous map $f:M\longrightarrow N$ is called an
{\bf $S$-morphism} from $M$ to $N$, if $f(sm)=sf(m)$, $\forall
s\in S$ and $\forall m\in M$. All left $S$-topological systems and all
$S$-morphisms between them constitute a category, denoted by
$S$-$\mathcal{TOP}$.
\end{Definition}
Clearly, if the semigroup $S$ contains zero element, then any
$S$-System $M$ must have an element $\theta$, such that
$s\theta=\theta$, $\forall s\in S$. If moreover, $M$ contains a
unique element $\theta_{M}$ satisfying $s\theta_{M}=\theta_{M}$,
$0m=\theta_{M}$, $\forall s\in S$ and $\forall m\in M$, we call
such $S$-System $M$ {\em central}. All central $S$-Systems and
$S$-morphisms between them also constitute a category. Clearly it
is a full sub-category of $S$-$\mathcal{TOP}$.

An important object in our study is quiver, whose definition is given as follows.

A {\em quiver} $\Gamma=(\Gamma_{0},\Gamma_{1})$ is an oriented
graph, where $\Gamma_{0}$ is the  set of the vertices and
$\Gamma_{1}$ is the set of arrows between vertices. A {\em
sub-quiver} of $\Gamma$ is just its oriented sub-graph.

We say a quiver $\Gamma$ is {\em a finite quiver} if $\Gamma_{0}$
and $\Gamma_{1}$ are both finite sets. We denote by $s:
\Gamma_{1}\rightarrow \Gamma_{0}$ and $t: \Gamma_{1}\rightarrow
\Gamma_{0}$ the maps, where $s(\alpha)=i$ and $t(\alpha)=j$ when
$\alpha:i\rightarrow j$ is an arrow from the vertex $i$ to the
vertex $j$.

A {\em path} $p$ in the quiver $\Gamma$ is either an ordered
sequence of arrows $p=\alpha_{n}\cdots\alpha_{2}\alpha_{1}$ with
$t(\alpha_{l})=s(\alpha_{l+1})$ for $1\leq l\leq n$, or the symbol
$e_{i}$ for $i\in \Gamma_{0}$. We call the path $e_{i}$ {\em
trivial path} and we define $s(e_{i})=t(e_{i})=i$. For a
non-trivial path $p=\alpha_{n}\cdots\alpha_{2}\alpha_{1}$, we
define $s(p)=s(\alpha_{1})$, and $t(p)=t(\alpha_{n})$.

A vertex $i$ in $\Gamma_{0}$ is called a {\em sink} if there is no
arrow $\alpha$ with $s(\alpha)=i$ and a {\em source} if there is
no arrow $\alpha$ with $t(\alpha)=i$.

Now we sketch a outline of this paper.

In Section 2, we introduce the
notion of topological representations of a quiver, as well as some
basic categorical notions such as exact sequences, injective (projective)
and indecomposable objects. As a result, we show the category of topological
representations of a quiver is equivalent to that defined in Definition
\ref{maindef} while $S$ being some path semigroup of the quiver (see Theorem \ref{The2.4}).

In Section 3, we investigate the relation between the category of topological
representations and that of linear representations of a quiver via  $P(\Gamma)$-$\mathcal{TOP}^o$ and $k\Gamma$-Mod. By using
the technics developed in \cite{L} concerning positively graded modules,
the results in Section 2 allow us to show that the property of the above two categories
are actually equivalent when the property is (positively) graded or vertex (positively) graded
(see Theorem \ref{topology}, Corollary \ref{Cor3.9} and \ref{Cor3.10}).

Section 4 is dedicated to discussing the homological theory of topological
representations of quivers. We introduce the concept of $\Gamma$-limit $Lim^{\Gamma}$ and use it to define the homology groups of topological
representations of quivers, and moreover we discuss the compatibility between $Lim^{\Gamma}$ and $H_{n}$ (cf. Theorem \ref{Cor4.17}).  We show that the homology groups of topological
representations can be well-defined by chain complexes (see Defintion \ref{Def4.13}). Moreover, some properties
of a quiver can be read from these homology groups (cf. Theorem \ref{Lem4.20}, Theorem \ref{The4.22} and \ref{The4.24}).

In Section 5, we discuss the homotopy theory of topological representations of quivers.
We define the homotopy equivalence between two morphisms in $\textbf{Top}\mathrm{-}\textbf{Rep}\Gamma$
(see Definition \ref{Def4.28}), and show that the parallel Homotopy Axiom also holds for top-representations based on the homotopy equivalence defined here (see Theorem \ref{The5.5}). Moreover the parallel Excision theorem holds for top-representations (see Theorem \ref{thm5.12}).

In Section 6, we mainly study the functor $At^{\Gamma}$ (see Definition \ref{Def6.2}) from the category $\textbf{Top}\mathrm{-}\textbf{Rep}\Gamma$ to the category
$\textbf{Top}$. We show that $At^{\Gamma}$ preserves homotopy equivalence between morphisms (see Theorem \ref{The6.4}). Moreover, we discuss some
properties of the functor $At^{\Gamma}$ (see Theorem \ref{The6.5}, \ref{The6.6} and \ref{The6.9}). Finally, we establish the relationship between the homotopy groups of a top-representation $(T,f)$ and the homotopy groups of $At^{\Gamma}(T,f)$ (see Corollary \ref{Cor6.11}).

\section{Top-representations of a quiver}

We start with the following definition.

\begin{Definition}\label{Def2.1}
%1.1
Let $\Gamma=(\Gamma_{0},\Gamma_{1})$ be a quiver with $\Gamma_{0}$
the set of vertices and $\Gamma_{1}$ the set of arrows between
vertices. A {\em top-representation} $(T,f)$ of a quiver
$\Gamma=(\Gamma_{0},\Gamma_{1})$ is a family of pairs of topological space $\{T_{i}:
i\in\Gamma_{0}\}$ together with continuous map $f_{\alpha}$:
$T_{i}\rightarrow T_{j}$ for each arrow $\alpha$: $i\rightarrow j$.
\end{Definition}

We also call the top-representation of the quiver $\Gamma$ defined in Definition \ref{Def2.1} the $\Gamma$-representation via topological spaces.

Let $(T, f)$ and $(T', f')$ be two topological representations of $\Gamma$.
A {\em morphism}  $h$: $(T,f)\rightarrow(T',f')$  between two
top-representations of $\Gamma$ is a collection of continuous maps
$\{h_{i}:T_{i}\rightarrow T'_{i}\}_{i\in\Gamma_{0}}$ such
that for each arrow $\alpha$: $i\rightarrow j$ in $\Gamma_{1}$ the following
diagram:

$$\xymatrix{
&T_i\ar[r]^{h_i}\ar[d]_{f_{\alpha}}&T_i^{'}\ar[d]^{f_{\alpha}^{'}}\\
&T_j\ar[r]^{h_j}&T_j^{'}
}$$
\;\;\;\;\;\;\;\;\;\;\;\;\;\;\;\;\;\;\;\;\;\;\;\;\;\;\;\;\;\;\;\;\;\;\;\;\;\;\;\;\;\;\;\;\;\;\;\;\;\;\;\;\;\;\;\;\;\;\;\;\;\;\;\;\;\;\;\;\;\;\;\;
\;\;\;$Figure(I)$
\\
commutes. If $h$: $(T,f)\rightarrow(T',f')$ and $g$:
$(T',f')\rightarrow(T'',f'')$ are two morphisms between
top-representations, then the composition $gh$ is defined to be
the collection of maps $\{g_{i}h_{i}$: $T_{i}\rightarrow
T''_{i}\}_{i\in\Gamma_{0}}$. In this way, we get {\em the category
of topological representations} of $\Gamma$, which we denote by $\textbf{Top}\mathrm{-}\textbf{Rep}\Gamma$.
%\wang{(For category {\bf
%Top-Rep}$(\Gamma)$, $T(i)\neq\emptyset,\forall i\in\Gamma_0,(T,f)\in${\bf
%Top-Rep}$(\Gamma)$)}.

If we think from any set $X$, there is a unique map
$\bar{0}:X\rightarrow\emptyset$ and to any set $Y$, there is a
unique map $\underline{0}:\emptyset\rightarrow Y$. Then, we can allow $T(i)=\emptyset$ in some special cases. For $(T,f)\in$$\textbf{Top}\mathrm{-}\textbf{Rep}\Gamma$, if $T(j)=\emptyset$ and there exists an arrow
$\alpha:i\longrightarrow j$, then $T(i)$ must also be an emptyset.
Then we can
define the zero object in $\textbf{Top}\mathrm{-}\textbf{Rep}\Gamma$ as follows:
$(T,f)$ is called the {\em zero object}, which we denote by
$(\emptyset,0)$, if $T(i)=\emptyset$ for all $i\in \Gamma_{0}$ and
$f_{\alpha}=1_{\emptyset}$ for each arrow $\alpha$ in
$\Gamma_{1}$.
For any two elements $(T,f),(T^{'},f^{'})$ in $\textbf{Top}\mathrm{-}\textbf{Rep}\Gamma$, if
$T^{'}(i)=\emptyset$ implies $T(i)=\emptyset$(called condition $\gamma$),
then $Hom_{\textbf{Top}\mathrm{-}\textbf{Rep}\Gamma}((T,f),(T^{'},f^{'}))=\{(h_i)_{i\in\Gamma}|f^{'}_{ij}h_i=h_jf_{ij},\forall\alpha:i\longrightarrow j\}$.
If the condition $\gamma$ doesn't hold, then we set $Hom_{\textbf{Top}\mathrm{-}\textbf{Rep}\Gamma}((T,f),(T^{'},f^{'}))=\{0\}$,
where $0$ is the zero morphism from $(T,f)$ to $(T^{'},f^{'})$.
For the object $(T,f)\in$$\textbf{Top}\mathrm{-}\textbf{Rep}\Gamma$,
$(T,f)$ is called a null object if there exists an $i_0\in\Gamma_0$
such that $T(i_0)=\emptyset$, otherwise, it is called a non-null object.

%%%%%%%%%%%%%%%%%%%%%%%%%%%%%%%%%%%%%%%%%%%%%%%%%%%%%%%%%%%%%%%%%%%%%%%%%%%%%%%%%%%%%%%%%%%%%%%%%%%%%%%%%%%%%%%%%%%%%%%%%%%%%%%%%%%%%%%%%%%%%%%%%%%%%%%%%%%%%%%%%%%%%%%%%%%%%%%%%%%%%%%%%%%%%%%%%%%%%%%%%%%%%%%%%%%%%%%%%%

If we think from any set $X$, there is a unique map
$\bar{0}:X\rightarrow\emptyset$ and to any set $Y$, there is a
unique map $\underline{0}:\emptyset\rightarrow Y$, then we can
define the zero object in {$\textbf{Top}\mathrm{-}\textbf{Rep}\Gamma$ as follows:
$(T,f)$ is called the {\em zero object}, which we denote by
$(\emptyset,0)$, if $T(i)=\emptyset$ for all $i\in \Gamma_{0}$ and
$f_{\alpha}=1_{\emptyset}$ for each arrow $\alpha$ in
$\Gamma_{1}$.

 An object $(S,f)$ is called a {\em sub-object} of an object $(T', f')$
in $\textbf{Top}\mathrm{-}\textbf{Rep}\Gamma$, if $T(i)\subseteq T'(i)$ for all
$i\in\Gamma_{0}$ and $f_{\alpha}=f'_{\alpha}|_{T(i)}$ for each
arrow $\alpha$ starting from $i$.

A {\em sum}, or coproduct, of two objects $(T,f)$ and $(T',f')$ in $\textbf{Top}\mathrm{-}\textbf{Rep}\Gamma$ is the object $(W,g)$, where $W(i)=T(i)\dot{\cup}
T'(i)$ for each $i\in\Gamma_{0}$ and $g_{\alpha}=f_{\alpha}\dot{\cup}
f'_{\alpha}$ for all $\alpha\in\Gamma_{1}$.
And, we denote the the sum of $(T,f)$ and $(T',f')$
as $(T,f)\coprod (T^{'},f^{'})$.
 An object $(T,f)$ is
said to be {\em indecomposable} if it can not be written as the
sum of any two nonzero top-representations. An object $(T,f)$ is
{\em simple} if it has no proper nonzero sub-objects. Clearly, a
simple object is indecomposable.

A product of two objects $(T,f)$ and $(T',f')$ in $\textbf{Top}\mathrm{-}\textbf{Rep}\Gamma$ is the object $(H,h)$,
where $H(i)=T(i)\times T^{'}(i)$ for each $i\in\Gamma_0$
and $h_{\alpha}=f_{\alpha}\times
f'_{\alpha}$ for all $\alpha\in\Gamma_{1}$.
And, we denote the the product of $(T,f)$ and $(T',f')$
as $(T,f)\times (T^{'},f^{'})$.

%It is easy to see that a
%representation $(T, f)$ is simple if $|T(i)|=1$ for each vertex
%$i\in\Gamma_{0}$.

Next, we illustrate with some examples.

\begin{Example}
  Let $(T, f)$ be  an object  in $\textbf{Top}\mathrm{-}\textbf{Rep}\Gamma$ and $V(i)=\{(a,a)\mid a \in
  T(i)\}$, $g_{\alpha}=(f_{\alpha}\times f_{\alpha})\mid
  _{V(i)}$ for each $i\in \Gamma_{0}$ and  an arrow $\alpha$ starting from $i$, then $(V,g)$ is a
  sub-object of $(T, f)\times (T, f)$, which is denoted as $1_{(T, f)}$.
  \end{Example}

\begin{Example}
Let $\Gamma$ be the quiver $1\cdot \rightarrow \cdot 2$, $(T, f)$
and $(T^{'}, f^{'})$ be two top-representations, where
$T(1)=\{x_{1},y_{1}\}$, $T(2)=\{x_{2},y_{2}\}$,
 $T^{'}(1)=\{x^{'}_{1},y^{'}_{1}\}$ and $T^{'}(2)=\{x^{'}_{2},y^{'}_{2}\}$
 are all spaces with discrete topology. Let $f_{\alpha}(x_{1})=x_{2}$,
 $f_{\alpha}(y_{1})=y_{2}$,
$f^{'}_{\alpha}(x^{'}_{1})=f^{'}_{\alpha}(y^{'}_{1})=x^{'}_{2}$, and let $h_{1}: T(1)\rightarrow T^{'}(1)$ with $h_{1}(x_{1})=y^{'}_{1}$
and $h_{1}(y_{1})=x^{'}_{1}$, $h_{2}: T(2)\rightarrow T^{'}(2)$
with $h_{2}(x_{2})=h_{2}(y_{2})=x^{'}_{2}$, then
$h=\{h_{1},h_{2}\}$ is a morphism from $(T, f)$ to $(T^{'}, f^{'})$.

\end{Example}

Let $P(\Gamma)$ be the set consisting of $0$ and all  paths in the
quiver $\Gamma$ . Define a multiplication $\cdot$ on $P(\Gamma)$
as follows: $0\cdot\rho=\rho\cdot 0=0$ for all $\rho\in
P(\Gamma)$, for any two paths $\rho_{ji}$ from $i$ to $j$ and
$\rho_{tk}$ from $k$ to $t$, $\rho_{ji}\cdot\rho_{tk}
= \left \{ \begin{array}{ll} 0, & \textmd{if}\;\; i\not=t\\
\rho_{ji}\rho_{tk}, & \textmd{if}\;\; i=t
\end{array}\right.$
where $\rho_{ji}\rho_{tk}$ means the connection of $\rho_{ji}$ and $\rho_{tk}$ for $i=t$.

Then $P(\Gamma)$ becomes a semigroup with zero $0$ under the
multiplication $\cdot$. Omitting $\cdot$, we usually write
 $\rho_{1}\rho_{2}$ instead of $\rho_{1}\cdot\rho_{2}$.

Now, by Definition \ref{maindef} we have a category $P(\Gamma)$-$\mathcal{TOP}$, further we can define a subcategory $P(\Gamma)$-$\mathcal{TOP}^{o}$ of it like this: the objects $M$ are
 $P(\Gamma)$-Systems satisfying (i) $P(\Gamma)M=M$; (ii) there is a unique element $\theta_{M}\in
 M$ such that $\{\theta_M$\} is a component of $M$ as a isolated point and \ $0m=\theta_{M}$, for all $m\in M$ (Here $\theta_{M}$
   acts as the ``{\em zero element}" of $M$); (iii) if $e_{i}m\neq
  \theta_{M}$, then $\alpha m\neq \theta_{M}$, for all arrows
  $\alpha$ starting from $i$.

Let $M\in $$P(\Gamma)$-$\mathcal{TOP}^{o}$, use $M_i$ denote the subspace of $M$ $e_iM\setminus\{\theta_M\}$
for all $i\in\Gamma_0$. Then, it is easy to show that $M=\dot{\cup}_{i\in\Gamma_0}M_i\dot{\cup}\{\theta_M\}$.
Since $e_i:M\rightarrow M,x\mapsto e_ix$ is continuous and $\{\theta_M\}$ is a close subset of
M, then $e_i^{-1}(\{\theta_M\})=M\setminus M_i$ is a close subset of $M$. Thus, we have $M_i$ is an open
subset of $M$ for each $i\in\Gamma_0$.

Note that for any $\rho\in P(\Gamma)$, we always have
 $\rho\theta_{M}=\rho(0m)=(\rho\cdot0)m=0m=\theta_{M}$. And
 clearly $P(\Gamma)$ is an object of
 $P(\Gamma)$-$\mathcal{TOP}^{o}$, called the {\em regular object}, and
 $\{\theta\}$ is the {\em zero object} of
 $P(\Gamma)$-$\mathcal{TOP}^{o}$, if we define the action as $\rho
 \theta=\theta$ for all $\rho \in P(\Gamma)$.

 For two objects $M$ and $N$ in $P(\Gamma)$-$\mathcal{TOP}^{o}$, a
{\em morphism} $\varphi$: $M\rightarrow N$ is defined as a continuous map
satisfying (i) $\varphi(\rho m)=\rho\varphi(m)$ for any $m\in M$
and $\rho\in P(\Gamma)$; (ii) $\varphi (m)\neq \theta_{N}$, if
$m\neq\theta_{M}$.

Note that (ii) is equivalent to saying $\varphi\big(M\setminus
\theta_{M}\big)\subseteq N\setminus \{\theta_{N} \}$, and when
$\rho=0$, from (i), it must hold that $\varphi(\theta_{M})=
\theta_{N}$.

And the object $M=\{\theta_M\}\in$$P(\Gamma)$-$\mathcal{TOP}^{o}$
acts as a zero object, we denote it as $\{0\}$.
%%%%%%%%%%%%%%%%%%%%%%%%%%%%%%%%%%%%%%%%%%%%%%%%%%%%%%%%%%%%%%%%%%%%%%%%%%%%%%%%%%%%%%%%%%%%%%%%%%%%%%%%%%%%%%%%%%%%%%%%%%%%%%%%%%%%%%%%%%%%%%%%%%%%%%%%%%%%%%%%%%%%%%%%%%%%%%%%%%%%%%%%%

For two objects $M$ and $N$ and a morphism
$\varphi$: $M\rightarrow N$ in $P(\Gamma)$-$\mathcal{TOP}^{o}$,
it is easy to show that $\varphi(e_iM\setminus\{\theta_M\})\subset e_iN\setminus\{\theta_N\}$, for all $i\in\Gamma_0$.
Thus, if $e_{i_0}N\setminus\{\theta_N\}=\emptyset$ but $e_{i_0}M\setminus\{\theta_M\}$ is nonempty for some $i_0\in\Gamma_0$,
which will lead to contradiction. In this case, we set
$Hom_{P(\Gamma)\mathrm{-}\mathcal{TOP}^{o}}(M,N)=\{0\}$, where the $0$ is the zero morphism from
$M$ to $N$. For any object $M\in$$P(\Gamma)$-$\mathcal{TOP}^{o}$,
if $e_{i_0}M\setminus\{\theta_M\}=\emptyset$ for some $i_0\in\Gamma_0$, it is called the null element, otherwise, it is called the non-null element.

Then, $P(\Gamma)$-$\mathcal{TOP}^{o}$ is exactly a subcategory of
the category $P(\Gamma)$-$\mathcal{TOP}$.

We have known from \cite{A}\cite{B} that, for a field $k$ and a
finite quiver $\Gamma$, there exists an equivalence between the
two categories {\bf Lin-Rep}$\Gamma$ and $k\Gamma$-Mod, where {\bf
Lin-Rep}$\Gamma$ is the category of $k$-linear representations of
$\Gamma$ and $k\Gamma$-Mod the $k\Gamma$-module category. It is
interesting for us to find that the similar result also holds
between the two weaker categories $\textbf{Top}\mathrm{-}\textbf{Rep}\Gamma$ and
$P(\Gamma)$-$\mathcal{TOP}^{o}$, that is, we have :

\begin{Theorem}\label{The2.4}
The two categories $\textbf{Top}\mathrm{-}\textbf{Rep}\Gamma$ and
 $P(\Gamma)$-$\mathcal{TOP}^{o}$ are equivalent.
\end{Theorem}
\begin{proof}: We start by defining two functors $F$: $\textbf{Top}\mathrm{-}\textbf{Rep}\Gamma$ $\rightarrow$ $P(\Gamma)$-$\mathcal{TOP}^{o}$ and
$H$: $P(\Gamma)$-$\mathcal{TOP}^{o}$ $\rightarrow$ $\textbf{Top}\mathrm{-}\textbf{Rep}\Gamma$.

 For an object $(T, f)$  in $\textbf{Top}\mathrm{-}\textbf{Rep}\Gamma$,
  set $M=\dot{\cup}_{i\in \Gamma_{0}} T(i)\cup\theta_{M}$,
  where $\theta_{M}$ is an element
 which is not in $T(i)$ for all $i\in \Gamma_{0}$.
  Define the action of $P(\Gamma)$ on the set $M$ as follows: for any $m \in M$,
 $\rho \in  P(\Gamma)$,

  (i) $\rho m =\theta_{M}$, if $\rho=0$;

  (ii) $\rho m=m$, if $m\in T(i)$ and $\rho=e_{i}$;

  (iii) $\rho m=f_{\alpha}(m)$, if $m\in T(i)$ and $\rho$ is an arrow $\alpha: i\rightarrow j$;

  (iv) $\rho m=f_{\alpha_{s}}\cdots f_{\alpha_{1}}(m)$, if $m \in T(i)$,
  $\rho=\alpha_{s}\cdots\alpha_{1}$ where $\alpha_{s},\cdots,\alpha_{1}$
 are arrows and $\alpha_{1}$ starts from $i$.

From  this definition, it is easy to see that
$P(\Gamma)\theta_{M}=\theta_{M}$, and that if $m\in T(i)$
but $\rho$ does not start from $i$, then $\rho
m=\rho(e_{i}m)=(\rho\cdot e_{i})m=0m=\theta_{M}$.

Clearly, $M$ is a $P(\Gamma)$-System under the action defined
above. Moreover, we can say $M$ is an object of
$P(\Gamma)$-$\mathcal{TOP}^{o}$.
  Firstly, the element $\theta_{M}$ satisfies $0m=\theta_{M}$ for all $m\in M$. And, obviously, $P(\Gamma)M\subseteq M$.
Conversely, for all $m\in M$, when $m\not=\theta_{M}$, suppose
$m\in T(i)$ for some $i$, then $m=e_{i}m$; when $m=\theta_{M}$, we
have
 $P(\Gamma) \theta_{M} = \theta_{M} $. Hence $M\subseteq P(\Gamma)M$.
  It follows that  $P(\Gamma)M=M$.
 If $e_{i}m\neq \theta_{M}$, which implies $m\in T(i)$, then for
 all arrows as $\alpha:i\rightarrow j$, $\alpha m=f_{\alpha}(m)\in
 T(j)$, so $\alpha m\neq \theta_{M}$.
 Then $M$ is
an object of $P(\Gamma)$-$\mathcal{TOP}^{o}$.

 Now, we can start to define the functor $F$: $\textbf{Top}\mathrm{-}\textbf{Rep}\Gamma$$\rightarrow P(\Gamma)$-$\mathcal{TOP}^{o}$ by $F(T, f)=M$.

Let $h$ be a morphism from $(T, f)$ to $(T', f')$ in the category
$\textbf{Top}\mathrm{-}\textbf{Rep}\Gamma$.  Then,  for each $i\in \Gamma_{0}$,  we
have a map  $h_{i} :$ $T(i)\rightarrow T'(i)$ satisfying the
Figure (I), i.e. $h_{j}f_{a} = f'_{a}h_{i}$ for each arrow
$\alpha$ from $i$ to $j$. It has been known that
$M=F(T, f)=\dot{\cup}_{i\in \Gamma_{0}}T(i)\cup\theta_{M}$ and
$M'=F(T', f')=\dot{\cup}_{i\in
\Gamma_{0}}T'(i)\cup\{\theta_{M'}\}$,
 Introducing a map
${\tilde h}: M\rightarrow M'$ satisfying that ${\tilde
h}|_{T(i)}=h_{i}$ for all $i$ and ${\tilde
h}(\theta_{M})=\theta_{M'}$. Thus, we can get $\tilde{h}(\alpha
m)=\alpha\tilde{h}(m)$ for each $m\in M$. Moreover, for each $m\in M$, then $\forall\rho=\alpha_s\cdots\alpha_2\alpha_1\in P(\Gamma), m\in M$, we have
\begin{align*}
\tilde{h}(\rho m)&=\tilde{h}\big(f_{\alpha_s}\cdots f_{\alpha_2}f_{\alpha_1}(m)\big)\\
&=h_k\big(f_{\alpha_s}\cdots f_{\alpha_2}f_{\alpha_1}(m)\big)\\
&=f'_{\alpha_s}h_{k-1}\big(f_{\alpha_{s-1}}\cdots f_{\alpha_2}f_{\alpha_1}(m)\big)\\
&=\cdots\\
&=f'_{\alpha_s}\cdots f'_{\alpha_2}f'_{\alpha_1}\big(\tilde{h}(m)\big)\\
&=\rho\big(\tilde{h}(m)\big)
\end{align*}
This shows that $F(h)=\tilde{h}$, and so $\tilde{h}$ is a morphism in $P(\Gamma)\textbf{-TOP}$. When $m\neq\theta_{M}$,
$\tilde{h}(m)\neq\theta_{M'}$ since
$\tilde{h}(T(i))=h_{i}(T(i))\subseteq T'(i)$.
 Therefore $\tilde{h}$ is a morphism from $M$ to $M'$. This means one can set $F(h)=\tilde{h}$.

We next want to define a functor $H$:
$P(\Gamma)$-$\mathcal{TOP}^{o}$ $\rightarrow$ $\textbf{Top}\mathrm{-}\textbf{Rep}\Gamma$. For an object $M$ in category of
$P(\Gamma)$-$\mathcal{TOP}^{o}$, let $T(i)=e_{i}M \setminus
\theta_{M}$. For all arrows $\alpha : i\rightarrow j$,  define
$f_{\alpha}:T(i)\rightarrow T(j)$ as follows: for all $m \in
T(i)$, suppose $m=e_{i}m^{'}$, let $f_{\alpha}(m)=\alpha m$, it is
well-defined since $\alpha m=\alpha (e_{i}m^{'})=\alpha m^{'}\neq
\theta_{M}$ and $\alpha m= (e_{j}\alpha)m=e_{j}(\alpha m)\in
T(j)$. Therefore let $H(M)=(T, f)$, where $S=\{T(i):i\in
\Gamma_{0}\}$, and $f=\{f_{\alpha}:$ there is an arrow $\alpha$
from $i$ to $j\}$. Then $H(M)$ is an object of category $\textbf{Top}\mathrm{-}\textbf{Rep}\Gamma$.

If $\varphi :M\rightarrow M'$ is a morphism  in
$P(\Gamma)$-$\mathcal{TOP}^{o}$, we have  $H(M)=(T, f)$,
$H(M')=(T', f')$, where $T(i)=e_{i}M \setminus \theta_{M}$ and
$T'(i)=e_{i}M' \setminus \{\theta_{M'}\}$. Since $\varphi
(e_{i}M)=e_{i}\varphi (M)\subseteq e_{i}M'$ and $\varphi(m)\neq
\theta_{M'}$ for all $m\in M$ and $m\neq \theta_{M}$, then we get
$\varphi_{i}:e_{i}M \setminus \theta_{M} \rightarrow e_{i}M'
\setminus \{\theta_{M'}\}$ by restriction, i.e.
$\varphi_{i}=\varphi|_{T(i)}:T(i)\rightarrow T'(i)$. For each
arrow $\alpha:i\rightarrow j$, we have $\alpha
\varphi(m)=\varphi(\alpha m)$, for all $m\in M$. So $\alpha
\varphi_{i}(m)=\varphi_{j}(\alpha m)$, for all $m\in T(i)$. Hence
$f'_{\alpha}\varphi_{i}(m)=\varphi_{j}f_{\alpha} (m)$, for all
$m\in T(i)$. Then, $f'_{\alpha}\varphi_{i}=\varphi_{j}f_{\alpha}$
for any arrow $\alpha:i\rightarrow j$.  Therefore we can set
$H(\varphi)=\{\varphi_{i}\}_{i\in \Gamma_{0}}$, which is a
morphism in $\textbf{Top}\mathrm{-}\textbf{Rep}\Gamma$.

Next, we will prove $F$ and $H$ are mutual-inverse equivalent
functors. Let $(T, f)$ be an object in $\textbf{Top}\mathrm{-}\textbf{Rep}\Gamma$, then
$M=F(T, f)=\dot{\cup}_{j\in \Gamma_{0}}T(j)\cup \{ \theta_{M}\}$
and $e_{i}M\setminus
\theta_{M}=e_{i}(\dot{\cup}_{j\in\Gamma_{0}}T(j))\setminus \{
\theta_{M}\}=e_{i}T(i)\setminus \theta_{M}=T(i)$. For an arrow
$\alpha :i\rightarrow j$ in $\Gamma_{1}$, the map
$f_{\alpha}:T(i)\rightarrow T(j)$ induces the map $\tilde
f_{\alpha}:F(T, f)\rightarrow F(T, f)$ satisfying $\tilde f_{\alpha}
(m)=\alpha m$  for all $m\in F(T, f)$. The restriction of $\tilde
f_{\alpha}$ on $e_{i}F(T, f)\setminus \theta_{M}=T(i)$ is just
$f_{\alpha}$. So $HF(T, f)=(T, f)$.

For a morphism $h=\{h_{i}\}_{i\in \Gamma_{o}}$: $(T, f)\rightarrow
(T', f')$, we have $F(h)=\tilde h$ where $\tilde
h\mid_{T(i)}=h_{i}$, $\tilde h(\theta_{M})=\theta_{M'}$. Due to
the definition of $H$, it follows  $HF(h)=\{h_{i}\}_{i\in
\Gamma_{0}}$. Thus, $HF={\bf id}$ the identity functor in $\textbf{Top}\mathrm{-}\textbf{Rep}\Gamma$.

Let $M$ be an object in $P(\Gamma)$-$\mathcal{TOP}^{o}$, then
$H(M)=(T, f)$, where $T(i)=e_{i}M\setminus \theta_{M}$ and
\[
f=\{f_{\alpha}:T(i)\rightarrow T(j)\mid\;\;
f_{\alpha}(m_{i})=\alpha m_{i}\; {\rm for\; an\; arrow}\;  \alpha:
i\rightarrow j \; {\rm and}\;m_{i}\in T(i)  \}.
\]
  When $i\neq j$, if there exists two elements $m, m^{'}\in M$, such that $e_{i}m=e_{j}m^{'}\neq
  \theta_{M}$, then for an arrow $\alpha: i\rightarrow k$, $\alpha(e_{i}m)=\alpha m\neq
  \theta_{M}$, but $\alpha(e_{j}m^{'})=(\alpha
  e_{j})m^{'}=0m^{'}=\theta_{M}$, this is a contradiction. Hence
  $T(i)\cap T(j)=\emptyset$ when $i\neq j$.
  So if we can
prove $M=\cup_{i\in \Gamma_{o}}T(i)\cup \theta_{M}$, then
$FH(M)=M$. In fact, $\cup_{i\in \Gamma_{o}}T(i)\cup
\theta_{M}\subseteq P(\Gamma)M=M$. Conversely, for all $m\in
M$, if $m=\theta_{M}$, it is clearly that $m\in \cup_{i\in
\Gamma_{o}}T(i)\cup \theta_{M}$, when $m\neq \theta_{M}$,
since $m\in M=P(\Gamma)M$, there is $\rho_{ji}\in P(\Gamma),m'\in
M$, such that $m=\rho_{ji}m'$. Clearly $m'\neq\theta_{M}$, so
$m=\rho_{ji}m'=e_{j}(\rho_{ji}m')\in e_{j}M\setminus
\theta_{M}=T(j)$. Therefore, $M\subseteq \cup_{i\in
\Gamma_{o}}T(i)\cup \theta_{M}$.

 For a morphism
$\varphi:M\rightarrow M'$, we have
$H(\varphi)=\{\varphi_{i}:T(i)\rightarrow T'(i)\mid\;\;
\varphi_{i}=\varphi\mid_{T(i)}\}_{i\in \Gamma_{0}}$. Moreover, due
to the definition of $F$, it follows $FH(\varphi)=\varphi$.
Therefore, $FH={\bf id}$ the identity functor in
$P(\Gamma)$-$\mathcal{TOP}^{o}$.

\end{proof}

\begin{Remark}

According to Theorem \ref{The2.4} and some known results in the theory of S-Systems of semigroups, it is easy for us to obtain the following conclusions:

(1)$P(\Gamma)$ is the coproduct of $P(\Gamma)e_i,i\in\Gamma_0$ in the category $P(\Gamma)$-$\mathcal{TOP}^{o}$.

(2)Let $M\in$$P(\Gamma)$-$\mathcal{TOP}^{o}$, and $M_{i_0}=e_{i_0}M\setminus\{\theta_M\}\neq\emptyset$ for some $i_0\in\Gamma_0$.
Then, for any $m_{i_0}\in M_{i_0}$, there exists a unique morphism $\varphi_{i_0}$ from $P(\Gamma)e_{i_0}$ to
$M$ such that $\varphi_{i_0}(e_{i_0})=m_{i_0}$.

(3)Let $M\in$$P(\Gamma)$-$\mathcal{TOP}^{o}$, and $M_{i}=e_iM\setminus\{\theta_M\}\neq\emptyset$ for all $i\in\Gamma_0$.
Then, for any $\{m_i|m_i\in M_i,i\in\Gamma_0\}$, there exists a unique morphism $\varphi$ from $P(\Gamma)$ to
$M$ such that $\varphi(e_{i})=m_{i}$ for all $i\in\Gamma$.

(4)For each $i\in\Gamma_0$, $P(\Gamma)e_i$ is a projective object in category $P(\Gamma)$-$\mathcal{TOP}^{o}$.

(5)The regular object $P(\Gamma)$ is a projective object in category $P(\Gamma)$-$\mathcal{TOP}^{o}$.
\end{Remark}

Due to Theorem \ref{The2.4} and the discussions of the product
of two objects $(T,f)$ and $(T^{'},f^{'})$ in category $\textbf{Top}\mathrm{-}\textbf{Rep}\Gamma$, for any two objects $M,N$ in category $P(\Gamma)$-$\mathcal{TOP}^{o}$,
the product of $M$ and $N$, denoted as $M\times N$, is $Q$,
where $Q_i=M_i\times N_i$
and $\rho(x,y)=(\rho x,\rho y)$ for all $i\in\Gamma_0,(x,y)\in M_i\times N_i,\rho\in P(\Gamma)$
($M_i=e_iM\setminus\{\theta_M\},N_i=e_iN\setminus\{\theta_N\},Q_i=e_iQ\setminus\{\theta_Q\}$).

In the category $\textbf{Top}\mathrm{-}\textbf{Rep}\Gamma$, for a morphism $h=\{h_{i}\}_{i \in  \Gamma_{0}}
  :(T, f)\rightarrow (T^{'}, f^{'})$, we define the image $Imh$ to be
  the suboject $(U,g)$ of $(T^{'}, f^{'})$, where $U(i)=Imh_{i}$ and
  $g_{\alpha}=f^{'}_{\alpha}|_{Imh_{i}}$ for each arrow $\alpha:i\rightarrow
  j$. We define the kernel $Kerh=(V,f^{''})$ as the fibre product of $(T, f)$ and itself under the morphism $h$, i.e., a sub-object  of
  $(T, f)\times (T, f)$ with $V(i)=\{(a,b)|a,b \in T(i)\, {\rm with }\, h_{i}(a)=
  h_{i}(b)\}$ and $f^{''}_{\alpha}=(f_{\alpha}\times
  f_{\alpha})|_{V(i)}$ for each arrow $\alpha:i\rightarrow j$.

If each $h_{i}$ is injective (respectively surjective), we call
$h$ a {\em monomorphism} (respectively an {\em epimorphism}), and
$h$ is an {\em isomorphism} if and only if $h$ is both monomorphic
and epimorphic. The morphism $h$ given in Example 2.2 is neither
monomorphic nor epimorphic. We call the sequence $(T, f)
\stackrel{h}{\rightarrow} (T^{'}, f^{'}) \stackrel
{h^{'}}{\rightarrow}(T^{''}, f^{''})$ a {\em related exact
sequence} if
$(Imh\times Imh)\bigcup 1_{(T^{'}, f^{'})}=Ker
h^{'}$.

Similarly, for a morphism $h:M\rightarrow N$ in category $P(\Gamma)$-$\mathcal{TOP}^{o}$
Kerh is a sub-object of $M\times M$ with $Kerh=\{(x_1,x_2|h(x_1)=h(x_2)\}$.
We call the sequence in category $P(\Gamma)$-$\mathcal{TOP}^{o}$ $Q
\stackrel{f}{\rightarrow} M \stackrel
{g}{\rightarrow}N$ a {\em related exact
sequence} if
$(Imf\times Imf)\bigcup 1_M=Kerg$.

 Then, we have

\begin{Proposition}
(i) The sequence $(\emptyset,0)\rightarrow (T, f)\stackrel
{h}{\rightarrow }(T^{'}, f^{'})$ is related exact if and only if
$h$ is a monomorphism.

(ii) Suppose $\mid T^{'}(i)\mid \geq 2$ for all $i\in \Gamma_{0}$,
then  the sequence $(T, f)\stackrel{h}{\rightarrow
}(T^{'}, f^{'})\rightarrow (\emptyset,0)$ is related exact if and
only if $h$ is an epimorphism.
\end{Proposition}

\begin{proof} (i) $(\emptyset,0)\rightarrow (T, f)\stackrel
{h}{\rightarrow }(T^{'}, f^{'})$ related exact

$\Longleftrightarrow$ $Kerh=1_{(T, f)}$

$\Longleftrightarrow$ $(Kerh)(i)=\{(a,a)|a\in T(i)\}$, $\forall
i\in \Gamma_{0} $

$\Longleftrightarrow$ $\{(a,b)|a,b\in
T(i),h_{i}(a)=h_{i}(b)\}=\{(a,a)|a\in T(i)\}$, $\forall i\in
\Gamma_{0} $

$\Longleftrightarrow$ $h_{i}(a)=h_{i}(b)$ implies $a=b$, $\forall
i\in \Gamma_{0}$ and  $a,b \in T(i)$

$\Longleftrightarrow$ $h_{i}$ is injective, $\forall i\in
\Gamma_{0} $

$\Longleftrightarrow$ $h$ is monomorphic.

(ii) $(T, f)\stackrel{h}{\rightarrow }(T^{'}, f^{'})\rightarrow
(\emptyset,0)$ related exact

$\Longleftrightarrow$ $(Imh\coprod Imh)\cup
1_{(T^{'}, f^{'})}=(T^{'}, f^{'})\coprod (T^{'}, f^{'})$

$\Longleftrightarrow$ $(Imh_{i}\coprod Imh_{i})\cup
\{(a^{'},a^{'})|a^{'}\in T^{'}(i)\}=T^{'}(i)\coprod T^{'}(i)$,
$\forall i\in \Gamma_{0}$

$\Longleftrightarrow$ if $a^{'},b^{'}\in T^{'}(i)$ and $a^{'}\neq
b^{'}$, then $(a^{'},b^{'})\in Imh_{i}\coprod Imh_{i}$, $\forall
i\in \Gamma_{0}$

$\Longleftrightarrow$ $h_{i}$ is surjective, $\forall i\in
\Gamma_{0}$

$\Longleftrightarrow$ $h$ is epimorphic.
\end{proof}

An object $(T, f)$ is said to be {\em projective} if for an
arbitrary epimorphism $h^{'}:(T^{'}, f^{'})\rightarrow
(T^{''}, f^{''})$, and an arbitrary morphism
$h^{''}:(T, f)\rightarrow (T^{''}, f^{''})$, there exists  a
morphism $h:(T, f)\rightarrow (T^{'}, f^{'})$ such that
$h^{''}=h^{'}h$, i.e. we have the commutative diagram

 $$\xymatrix{
 &&(T,f)\ar[ld]_{h}\ar[d]^{h^{''}}\\
 &(T^{'},f^{'})\ar[r]^{h^{'}}&(T^{''},f^{''})
 }$$
 \;\;\;\;\;\;\;\;\;\;\;\;\;\;\;\;\;\;\;\;\;\;\;\;\;\;\;\;\;\;\;\;\;\;\;\;\;\;\;\;\;\;\;\;\;\;\;\;\;\;\;\;\;\;\;\;\;\;\;\;\;\;\;\;\;\;\;\;\;\;\;\;
\;\;\;$Figure(II)$

 Dually, an object $(T, f)$ is said to be {\em injective}, if for an
 arbitrary monomorphism $h^{'}:(T^{''}, f^{''})\rightarrow (T^{'}, f^{'})$
 and an arbitrary morphism $h^{''}:(T^{''}, f^{''})\rightarrow
 (T, f)$, there exists a morphism $h:(T^{'}, f^{'})\rightarrow
 (T, f)$, such that $h^{''}=hh^{'}$, i.e. we have the commutative diagram

 $$\xymatrix{
 &(T,f)&\\
 &(T^{''},f^{''})\ar[u]^{h^{''}}\ar[r]^{h^{'}}&(T,f)\ar[lu]_{h}
 }$$
 \;\;\;\;\;\;\;\;\;\;\;\;\;\;\;\;\;\;\;\;\;\;\;\;\;\;\;\;\;\;\;\;\;\;\;\;\;\;\;\;\;\;\;\;\;\;\;\;\;\;\;\;\;\;\;\;\;\;\;\;\;\;\;\;\;\;\;\;\;\;\;\;
\;\;\;$Figure(III)$

 As a corollary, the following holds naturally:

\begin{Corollary}
(i) An object $(V,f)$ in the category $\textbf{Top}\mathrm{-}\textbf{Rep}\Gamma$ is
projective (respectively injective, simple, indecomposable) if and
only if $F(V,f)$ is projective (respectively injective, simple,
indecomposable) in the category $P(\Gamma)$-$\mathcal{TOP}^{o}$;

(ii) A sequence $(U,f)\rightarrow (V,g)\rightarrow (W,h)$ in the
category $\textbf{Top}\mathrm{-}\textbf{Rep}\Gamma$ is related exact if and only if the
induced sequence $F(U,f)\rightarrow F(V,g)\rightarrow F(W,h)$ is
related exact in the category $P(\Gamma)$-$\mathcal{TOP}^{o}$.
\end{Corollary}

\section{Relations between top-representations and linear-representations of quivers }

In this section, every semigroup
mentioned contains a zero element
  and the quiver $\Gamma$ is finite.

 A semigroup S is {\em graded} if there exists a family of
non-empty subsets $\{S_{(i)}\}_{i\in \textbf{Z}}$, where $S_{(0)}$
is a sub-semigroup, $S=\cup _{i\in \textbf{Z}}S_{(i)}$,
$S_{(i)}S_{(j)}\subseteq S_{(i+j)}$, and $S_{(i)}\cap
S_{(j)}=\{0\}$ for $i\neq j$. When $S=\cup _{i\geq 0}S_{(i)}$ is a
graded semigroup, S is called {\bf positively graded}. And a
positively graded semigroup S is called {\bf strongly graded} if
$S_{(i)}S_{(j)}= S_{(i+j)}$ for any $i,j\geq 0$.

Note that, the path semigroup $P(\Gamma)$ which consists of $0$
and all paths in $\Gamma$ has a natural positive gradation:
$P(\Gamma)=\cup _{i\geq 0}(P(\Gamma))_{(i)}$, where
$(P(\Gamma))_{(i)}$ consists of 0 and all the paths whose length
is $i$. This positive gradation of $P(\Gamma)$ is strongly graded
obviously.

 Let $S$ be a graded semigroup, $M$ be a $S$-System in
$S$-$\mathcal{TOP}^{o}$, if there exists a family of nonempty
subsets $\{M_{(i)}\}_{i\in \textbf{Z}}$, such that $M=\cup _{i\in
\textbf{Z}}M_{(i)}$, $S_{(i)}M_{(j)}\subseteq M_{(i+j)}$, and
$M_{(i)}\cap M_{(j)}=\theta_{M}$ for $i\neq j$, then $M$ is said
to be {\bf graded}. Similarly, for a positively graded semigroup $S$,
we can give the definition of positive gradation for $M$.

 Let $M$ be a positively graded $P(\Gamma)$-system in
$P(\Gamma)$-$\mathcal{TOP}^{o}$, where the gradation of
$P(\Gamma)$ is natural. If every homogeneous component is the
union of some $M_{\upsilon}=e_{\upsilon}M$, that is, for every
vertex $\upsilon \in \Gamma_{0} $, $e_{\upsilon}M$ is contained in
some a homogeneous component, then $M$ is said to be {\bf vertex
positively graded}.
 Clearly, if $M$ is positively graded and for any  vertex $\upsilon \in \Gamma_{0} $,
 $M_{\upsilon}$ contains at most one element except for
 $\theta_{M}$, then $M$ is vertex positively graded.

\begin{Definition}(\cite{LC}) (i) Function $F:\Gamma_{0}\longrightarrow \textbf{Z}
$ is called an {\em arrow function} on $\Gamma$, if $F(t(\alpha))=
F(s(\alpha))+1$ for any arrow $\alpha \in \Gamma_{1}$\emph{(see
[4])}.

(ii) If function $F:\Gamma_{0}\longrightarrow
\textbf{Z}^{+}\cup\{0\}$ is an arrow function on $\Gamma$, we call
$F$ an {\em arrow positive function} on $\Gamma$.
\end{Definition}

\begin{Proposition}(\cite{LC})  F is an arrow positive function on a connected quiver
$\Gamma$, $G:\Gamma_{0}\longrightarrow \textbf{Z}^{+}\cup \{0\}$
is another positive function, then G is an arrow positive function
on $\Gamma$ if and only if there exists an integer $k$, such that
$F=G+k$.\end{Proposition}

\begin{Definition}(\cite{LC})  For a non-trivial path $\rho$ in a quiver $\Gamma$, if $s(\rho)=e(\rho)$, we
say it is an  $\textbf{oriented cycle}$. A sub-quiver $\Delta$ of a
quiver $\Gamma$ is said to be a $\textbf{ cycle}$ , if when omitted the
direction of all arrows, the graph, which we call the base graph,
is closed . In a cycle , when the number of clockwise arrows
equals to the number of anti-clockwise arrows, we say the cycle is
 $\textbf{symmetric}$.\end{Definition}

\begin{figure}
\begin{picture}(100,100)(0,0)
\put(170,50){\makebox(0,0){$1\;\bullet$}}
\put(175,52){\line(1,1){25}}\put(202,80){\makebox(0,0){$2\;\bullet$}}
\put(210,80){\line(1,0){30}}\put(250,80){\makebox(0,0){$3\;\bullet$}}
\put(260,70){\makebox(0,0){$\cdot$}}\put(265,65){\makebox(0,0){$\cdot$}}\put(270,60){\makebox(0,0){$\cdot$}}
\put(265,50){\makebox(0,0){$n-2\;\bullet$}}\put(275,45){\line(-1,-1){20}}
\put(265,20){\makebox(0,0){$\bullet\;n-1$}}\put(240,20){\line(-1,0){30}}
\put(200,20){\makebox(0,0){$n\;\bullet$}}\put(198,25){\line(-1,1){20}}
\end{picture}
\begin{center}
{$Figure (IV)$}
\end{center}
\end{figure}

By Definition 3.1 and 3.2, when a quiver has no cycle, we can
always define an arrow positive function on it. And it is clearly
that, an oriented cycle is not symmetric. Indeed, we have the
following conclusions.
\begin{Proposition}(\cite{LC})  Assume  $\Gamma$ is a finite connected quiver. Then there
exists an arrow positive function on $\Gamma$ if and only if $\Gamma$
does not contain any non-symmetric cycle.\end{Proposition}

\begin{Lemma}   For a finite quiver $\Gamma$, if all $P(\Gamma)$-Systems in
$P(\Gamma)$-$\mathcal{TOP}^{o}$ are positively graded, then any
cycle in $\Gamma$ is symmetric.\end{Lemma}

\begin{proof}:  Suppose $\Gamma$ contains a cycle $\Delta$ with the
base graph like $Figure (IV)$,
 consider a special $P(\Gamma)$-System $M$ in $P(\Gamma)$-$\mathcal{TOP}^{o}$, its
 top-representation according to the equivalence in Theorem 2.3 is $(T, f)$, where all $S(\upsilon)$
 are equal and contain only one element, the maps between them are
 all identity maps.

 Since $M$ is positively graded, from its special construction it is also vertex positively graded.
Define a function $F:\Gamma_{0}\longrightarrow \textbf{Z}^{+}\cup
\{0\}$ as follows: $F(\upsilon)=i$, if $e_{\upsilon}M\subseteq
M_{(i)}$. It is easy to know $F$ is an arrow positive function on
$\Gamma$, and so it is on $\Delta$. Indeed, if for an arrow
$\alpha:\upsilon\rightarrow\omega$, $F(\upsilon)=i$, then from the
construction of $M$, $e_{\omega}M=\alpha (e_{\upsilon}M)\subseteq
\alpha M_{(i)}\subseteq P(\Gamma)_{(1)}M_{(i)}\subseteq
M_{(i+1)}$, i.e. $F(\omega)=F(\upsilon)+1$. By Lemma 3.2, $\Delta$
is a symmetric cycle.
\end{proof}

Thus, we get the main result of this section:

\begin{Theorem}\label{topology}  Suppose $\Gamma$ is a finite connected quiver.
The following properties are equivalent:

(i) all $P(\Gamma)$-Systems in $P(\Gamma)$-$\mathcal{TOP}^{o}$ are
positively graded;

(ii) any cycle in $\Gamma$ is symmetric;

(iii) there exists an arrow positive function on $\Gamma$;

(iv) all $P(\Gamma)$-Systems in $P(\Gamma)$-$\mathcal{TOP}^{o}$
are vertex positively graded.

%(v) the regular object in $P(\Gamma)$-$\mathcal{TOP}^{o}$ is
%vertex positive graded ???????????????
\end{Theorem}

\begin{proof}: (i)$\Rightarrow$(ii): By Lemma  3.5.

(ii)$\Rightarrow$(iii): By Lemma 3.3.

(iii)$\Rightarrow$(iv):  Suppose $F:\Gamma_{0}\longrightarrow
\textbf{Z}^{+}\cup \{0\}$ is an arrow positive function on quiver
$\Gamma$, since for any $P(\Gamma)$-System $M$ in
$P(\Gamma)$-$\mathcal{TOP}^{o}$, $M=\cup_{\upsilon\in
\Gamma_{0}}e_{\upsilon}M$, let $M_{(i)}=\cup_{\upsilon\in
\Gamma_{0},F(\upsilon)=i}e_{\upsilon}M$, then
$M=\cup_{F(\upsilon)=i,\upsilon\in \Gamma_{0}} M_{(i)}$ is a
positive gradation. Actually, for any arrow $\alpha$ in
$\Gamma_{1}$, we have $F(t(\alpha))=F(s(\alpha))+1$. Then when
$i\neq F(s(\alpha))$, $\alpha M_{(i)}=\theta_{M}\subseteq
M_{(i+1)}$, since $\alpha M=\alpha e_{s(\alpha)}M$. And from the
definition of $M_{(i)}$ and $\alpha M_{(F(s(\alpha)))}\subseteq
e_{t(\alpha)}M\subseteq M_{(F(t(\alpha)))}=M_{(F(s(\alpha))+1)}$,
we know $M$ is vertex positively graded.

(iv)$\Rightarrow$(i): By the definition of vertex positively
graded.
\end{proof}

\begin{Example}
Let $\Gamma$ be a quiver as
$1\stackrel{\alpha_{1}}{\rightarrow}2\stackrel{\alpha_{2}}{\rightarrow}3\stackrel{\alpha_{3}}{\rightarrow}\cdots\stackrel{\alpha_{n-1}}{\rightarrow}n$. Then all $P(\Gamma)$-Systems in $P(\Gamma)$-$\mathcal{TOP}^{o}$
are positively graded.
\end{Example}
 Indeed, if $M$ is a $P(\Gamma)$-System of $P(\Gamma)$-$\mathcal{TOP}^{o}$,
 let $M_{(i)}=e_{i}M$, $i=1, 2, \cdots, n$. From the proof of Theorem 2.3, we know $M=\cup_{i=1}^{n}(e_{i}M\setminus \theta_{M})\cup
 \theta_{M}$ and $(e_{i}M\setminus \theta_{M})\cap (e_{j}M\setminus
 \theta_{M})=\emptyset$ when $i\neq j$, so
 $M=\cup_{i=1}^{n}M_{(i)}$ and $M_{(i)}\cap M_{(j)}=\theta_{M}$
 when $i\neq j$. And the inclusion $(P(\Gamma))_{(i)}M_{(j)}\subseteq
 M_{(i+j)}$is also easy to prove, here $M_{(i)}=\theta_{M}$
 for any $i>n$.

In final, we discuss the relation between the
category of top-representations and that of linear-representations.

\begin{Theorem}\label{module}(\cite{W})$\;$
Let $\Gamma$ be a finite connected quiver with  the
corresponding path algebra $k\Gamma$. Then the following statements are
equivalent:

(i) all $k\Gamma$-modules are graded;

(ii) any cycle in $\Gamma$ is symmetric;

(iii) there exists an arrow function on $\Gamma$;

(iv) all $k\Gamma$-modules are vertex graded.

%(5) the regular $k\Gamma-$modules is vertex graded.
\end{Theorem}

By Theorem \ref{topology} and Theorem \ref{module}, we know that for a finite
connected quiver $\Gamma$, there is an arrow function on it if and
only if there is an arrow positive function on it. Since all the
proofs in [4] about gradation were based on positive gradation,
all conclusions about gradation in [4] can be equivalently
replaced by the ones about positive gradation. Similarly, our
results about positive gradation in this paper can be equivalently
replaced by the ones about gradation. Then through the common
statement (ii) in Theorem \ref{topology} and Theorem \ref{module}, we have a
collection of equivalent statements. In particular, we have the
following corollaries:

%Note that in Theorem 3.6 and Theorem 3.7, the states (ii) are the
%same ones, hence all states in these two theorems are equivalent
%on a given finite connected quiver $\Gamma$. In particular, we
%have the following corollaries:

%\begin{theorem}
%Let $\Gamma$ be a finite connected quiver and $k$ a field, then
%the following properties are equivalent:

%(1) any cycle in $\Gamma$ is symmetric,

%(2) all f.d.$k\Gamma-$modules are graded,

%(2)$^{'}$ all objects in $P(\Gamma)$-$\mathcal{TOP}^{o}$ are
%positive graded,

%(3) there exists an arrow function on $\Gamma$,

%(3)$^{'}$ there exists an arrow positive function on $\Gamma$,

%(4) all f.d.$k\Gamma-$modules are vertex graded,

%(4)$^{'}$ all objects in $P(\Gamma)$-$\mathcal{TOP}^{o}$ are
%vertex positive graded.

%(5) the regular $k\Gamma-$modules is vertex graded.

%(5)$^{'}$ the regular object in $P(\Gamma)$-$\mathcal{TOP}^{o}$ is
%vertex positive graded.
%\end {theorem}
%\#

\begin{Corollary}\label{Cor3.9}
Let $\Gamma$ be a finite connected quiver and $k$ a field, then
the following statements are equivalent:

(i) all $k\Gamma$-modules are (positively) graded;

(ii) all $P(\Gamma)$-Systems in $P(\Gamma)$-$\mathcal{TOP}^{o}$
are (positively) graded.
\end{Corollary}

%\begin{corollary}Let $\Gamma$ be a finite connected quiver , then
%the following two properties are equivalent:

%(i) there exists an arrow function on $\Gamma$;

%(ii) there exists an arrow positive function on $\Gamma$.
%\end{corollary}

\begin{Corollary}\label{Cor3.10}
Let $\Gamma$ be a finite connected quiver and $k$ a field, then
the following statements are equivalent:

(i) all $k\Gamma$-modules are vertex (positively) graded;

(ii) all $P(\Gamma)$-Systems in $P(\Gamma)$-$\mathcal{TOP}^{o}$
are vertex (positively) graded.
\end{Corollary}

%Note: Because I think there exists some questions in the proof of
%Lemma 3.6 in [4] at $(5)\Rightarrow(3)$, we do not deduce the
%corresponding result as $(5)$ in our Theorem 3.6.

From the two theorems above, we find that on a finite connected
quiver $\Gamma$, there are some interesting relations between the
 categories $P(\Gamma)$-$\mathcal{TOP}^{o}$ and $k\Gamma$-Mod,
one of which is not abelian while the other is. Since the
top-representation category $\textbf{Top-Rep}\Gamma$ is equivalent to
the category $P(\Gamma)$-$\mathcal{TOP}^{o}$, and the
linear-representation category {\bf Lin-Rep}$\Gamma$ is equivalent
to the category $k\Gamma$-Mod, there are also some similar
relations between the two representation categories $\textbf{Top-Rep}\Gamma$ and {\bf Lin-Rep}$\Gamma$.

%%%%%%%%%%%%%%%%%%%%%%%%%%%%%%%%%%%%%%%%%%%%%%%%%%%%%%%%%%%%%%%%%%%%%%%%%%%%%%%%%%%%%%%%%%%%%%%%%%%%%%%%%%%%%%%%%%%%%%%%%%%%%%%%%%%%%%%%%
%%%%%%%%%%%%
%%%%%%%

\section{Homology groups of top-representations}

In \cite{HJ}, for a quiver $\Gamma$,  Henrik Holm and P. Jorgensen introduced a $\Gamma$-bundle associated to two chains.
In \cite{CJ}, Bustamante J C, Dionne J and Smith D discussed homology theories for oriented algebras.
In this section we hope to associate a quiver to chain complexes and use the homology groups of these chain complexes to reflect the quiver.

We recall from \cite{AT} some basic notations in algebraic topology. For any topological space and each integer $n\geq0$,  $S_n(X)$ denoted as the free abelian group with basis all continuous maps from $\Delta^n$ to $X$.  For any positive integer $n$, there are $n+1$ continuous maps $\varepsilon_i^n(0\leq i\leq n)$ from $\Delta^{n-1}$ to
$\Delta^{n}$($\Delta^n$ is the standard n-simplex), defined as
\begin{equation}
\varepsilon_i^n(t_0,t_1,\dots,t_{n-1})=(t_0,t_1,\dots,t_{i-1},0,t_{i+1}\dots,t_n)
\end{equation}
Then for each $n\geq0$, boundary operator $\partial_n$ from $S_n(X)$ to $S_{n-1}(X)$, defined as
\begin{equation}
\partial_n(f)=\sum_{i=1}^{i=n}(-1)^if\varepsilon_i^n
\end{equation}
 where $f$ is a continuous map from $\Delta^n$ to $X$, $S_{-1}(X)=0$, and $\partial\partial=0$. $S(X)$ denotes the following chain complex:
  \begin{equation}
  \xymatrix{ &\cdots\ar[r]&S_2(X)\ar[r]^{\partial_2}&S_1(X)\ar[r]^{\partial_1}&S_0(X)\ar[r]&0
  }
  \end{equation}

  If $f:X\rightarrow Y$ is countinuous and if $\alpha:\Delta^n\rightarrow X$ is an n-simplex in $X$, then $f\alpha$ is an n-simplex in $Y$. Extending by linearity gives a homomorphism $S_n(f):S_n(X)\rightarrow S_n(Y)$. Then, define $(S(f))_{n}=S_n(f)$ for all $n\geq-1$. Obviously, $S$ is a functor from $\textbf{Top}$ to $\textbf{C}$, where $\textbf{Top}$ is the category of topological spaces with  continuous functions as morphisms and $\textbf{C}$ is the category of chain complexes of abelian groups with chain maps as morphisms.
whose object is a chain complex of abelian groups,  written as $\mathop{A}\limits^{\bullet}$.

  For a continuous map between two topological spaces $f:X\rightarrow Y$, $S(f)$ denoted as $f_*$.
Then for each $n\geq0$, $H_n(X)=H_n(S(X))$. And, $H_n(f)=H_n(f_*):H_n(X)\rightarrow H_n(Y)$. It is easy to show $H_n$ is a functor from
$\textbf{Top}$ to $\textbf{Ab}$ for each $n\geq0$. For the emptyset $\emptyset$, we set $S(\emptyset)=0$. The coset $z_n+B_n(X)$, where $z_n$ is an $n$-cycle, is called the {\bf homology class of $z_n$}, and is denoted as $clsz_n$  (see \cite{AT}).

We further introduce some notations. Let $\textbf{C}\mathrm{-}\textbf{Rep}\Gamma$ be the category of $\Gamma$-representations via chain complexes of abelian groups, that is, for any $\Gamma$-representation via chain complexes of abelian groups,
 $(\mathop{C}\limits^{\bullet},f)\in\textbf{C}\mathrm{-}\textbf{Rep}$, $\mathop{C}\limits^{\bullet}(i)$ is a chain complex of abelian groups and $f_{ij}$ is the chain morphism for any arrow $i\rightarrow j$ in $\Gamma$.

  Let $\textbf{Ab}\mathrm{-}\textbf{Rep}\Gamma$ be the category of $\Gamma$-representations via abelian groups, that is, for
 any $\Gamma$-representation via abelian groups $(G,f)\in\textbf{Ab}\mathrm{-}\textbf{Rep}\Gamma$, $G(i)$ is an abelian group and $f_{ij}$ is a group homomorphism between abelian groups  for any $i\rightarrow j$ in $\Gamma$.

 It is easy to see that the only difference of $\textbf{Ab}\mathrm{-}\textbf{Rep}\Gamma$ with the category $\textbf{Lin}\mathrm{-}\textbf{Rep}\Gamma$ of linear representations of $\Gamma$ is that each vertex $i$ of $\Gamma$ is configured with an abelian group $G(i)$, which does not need to be a linear space. Then considering the path ring $\mathbb Z\Gamma$ over the integer ring $\mathbb Z$, we can obtain the similar result with the categorical equivalence between $\textbf{Lin-Rep}\Gamma$ and $k\Gamma$-Mod over a field $k$ as follows, using the similar mutual invertible functors.
  \begin{Proposition}
The categories $\textbf{Ab-Rep}\Gamma$ and $\mathbb Z\Gamma$-Mod are equivalent.
\end{Proposition}
Note that the path ring $\mathbb Z\Gamma$ is indeed the semigroup ring $\mathbb ZP(\Gamma)$ of the semigroup $P(\Gamma)$.

According to \cite{HL},
for a complex chain of abelian groups:
$$\mathop{G}\limits^{\bullet}:\; \cdots\xrightarrow{f_{n+1}} G_n\xrightarrow{f_n} G_{n-1}\xrightarrow{f_{n-1}} \cdots,$$  $\mathop{lim}\limits_{\leftarrow}G_n=\{(g_n)_{n\in\mathds{Z}}\in\prod_{n\in\mathds{Z}}G_n|f_n(g_n)=g_{n-1},\forall n\in\mathds{Z}\}$ is a a subgroup of $\prod_{n\in\mathds{Z}}G_n$.

We can think the complex chain $\mathop{G}\limits^{\bullet}$ to be on the infinite linear quiver $A_{\infty}$. In follows, we think an (infinite) quiver to be a generalization of (infinite) linear quiver. So, as a generalization of                                       $\mathop{lim}\limits_{\leftarrow}$, we introduce the notation $lim^{\Gamma}$.
\begin{Definition}\label{sequence}
Let $\Gamma$ be a quiver, and $(G,f)$ be a $\Gamma$-representation via abelian groups. Define
$$lim^{\Gamma}(G,f)=\{(x_i)_{i\in\Gamma}: \;x_i\in G(i)\;\text{and}\;f_{ij}(x_i)=x_j\;\text{if}\;f_{ij}\;\text{exists} \}\subseteq \prod_{i\in \Gamma}G(i).$$
We call $lim^{\Gamma}(G,f)$ the {\bf $\Gamma$-limit of $(G,f)$}.
\end{Definition}

 Note that $lim^{\Gamma}(G,f)$ is a subgroup of $\prod_{i\in \Gamma}G(i)$, since for any $(x_i)_{i\in\Gamma},(y_i)_{i\in\Gamma}\in lim^{\Gamma}(G,f)$, we have $(x_i)_{i\in\Gamma}-(y_i)_{i\in\Gamma}=
(x_i-y_i)_{i\in\Gamma}$ and $f_{ij}(x_i-y_i)=f_{ij}(x_{i})-f_{ij}(y_{i})=x_{j}-y_{j}$ for all arrows $i\rightarrow j\in\Gamma_{1}$, therefore $(x_i)_{i\in\Gamma}-(y_i)_{i\in\Gamma}\in lim^{\Gamma}(G,f)$.

\begin{Definition}
 Let $\Gamma$ be a quiver, and $(\mathop{C}\limits^{\bullet},f)$ be a $\Gamma$-representation via
 chain complexes of abelian groups. Define $\mathop{lim^{\Gamma}}\limits^{\bullet}(\mathop{C}\limits^{\bullet},f)$ as a subcomplex of $\prod_{i\in\Gamma}\mathop{C}\limits^{\bullet}(i)$, such that
$(\mathop{lim^{\Gamma}}\limits^{\bullet}(\mathop{C}\limits^{\bullet},f))_n=lim^{\Gamma}(C^n,f^n)$,  where each
 $(C^n,f^n)$ is a $\Gamma$-representation via abelian groups and  $C^n(i)=\mathop{C}\limits^{\bullet}(i)_n$, $(f^n)_{ij}=(f_{ij})_n$.
\end{Definition}

In the case that $\Gamma$ is an acyclic quiver.
Then $\Gamma_{0}$ is a partially ordered set with the partial order $\geq$ defined as follows:
 for any two vertices $i,j\in\Gamma_{0}$, we define $i\geq j$ if there is a path from $i$ to $j$. In particular,  for any vertex $i\in\Gamma_{0}$, we say that
there is a path of length zero from $i$ to $i$, i.e.  $i\geq i$.

For $(G,f)\in\textbf{Ab}\mathrm{-}\textbf{Rep}\Gamma$ and
 any $i\geq j\in\Gamma_{0}$, if $i=j$ we set $\theta_{ij}=Id_{G(i)}$;
otherwise, there is at least a path $i\rightarrow i_{0}\rightarrow i_{1}\rightarrow \dots\rightarrow i_{n}\rightarrow j$
and then we set $\theta_{ij}=f_{i_{n},j}f_{i_{n-1},i_{n}}\dots f_{i,i_{0}}$.

Note that for  any $i\geq j\in\Gamma_{0}$, if there are more than one path from $i$ to $j$,
then for any two paths
$$i\rightarrow i_{0}\rightarrow i_{1}\rightarrow\dots\rightarrow i_{n}\rightarrow j\;\;\text{ and }\;\; i\rightarrow j_{0}\rightarrow j_{1}\rightarrow\dots\rightarrow j_{m}\rightarrow j,
$$
it needs to assume that $f_{j_{m},j}f_{j_{m-1},j_{m}}\dots f_{j_0,j_1}f_{i,j_0}=f_{i_{n},j}f_{i_{n-1},i_{n}}\dots f_{i_0,i_1}f_{i,i_0}$ so as to both be defined as  $\theta_{ij}$.

Thus, in the meaning of \cite{RJ},
 $(G(i),\theta_{ij})_{i\in\Gamma_{0},i\geq j}$ is an inverse system  in the category \textbf{Ab} over the partially ordered set $\Gamma_{0}$  and $lim^{\Gamma}(G,f)=\mathop{lim}\limits_{\leftarrow}(G(i),\theta_{ij})$.

Similarly, for any $(\mathop{C}\limits^{\bullet},f)\in \textbf{C}\mathrm{-}\textbf{Rep}\Gamma$, $(\mathop{C}\limits^{\bullet}(i),\theta_{ij})$
is an inverse system in the category \textbf{C} over the partially ordered set $\Gamma_{0}$  and $\mathop{lim^{\Gamma}}\limits^{\bullet}(\mathop{C}\limits^{\bullet},f))=\mathop{lim}\limits_{\leftarrow}(\mathop{C}\limits^{\bullet}(i),\theta_{ij})$.

Note that we cannot define the partial order as above for a quiver $\Gamma$ containing an oriented cycle: otherwise, we have $i\not=j$ but  $i\geq j,j\geq i$, which is a contradiction.

 Conversely, let $(M_{i},\theta_{ij})_{i,j\in I,i\geq j}$ be an inverse system in the category \textbf{Ab} over the partially ordered set $I$. Then we can define the quiver $\Gamma_{I}=(\Gamma_{0},\Gamma_{1},s,t)$, where $\Gamma_{0}=I,\Gamma_{1}=\{i\rightarrow j|i,j\in I\;\text{and}\;i\geq j\}$. Let $(G,f)\in\textbf{Ab}\mathrm{-}\textbf{Rep}\Gamma_I$, and $G(i)=M_{i},\forall i\in\Gamma_{0}$
$f_{ij}=\theta_{ij}, \forall i\rightarrow j$ in $\Gamma_1$, then we have $lim^{\Gamma_I}(G,f)=\mathop{lim}\limits_{\leftarrow}(M_{i},\theta_{ij})$ in the category $\textbf{Ab}\mathrm{-}\textbf{Rep}\Gamma_I$.

The converse discussion holds for an inverse system $(M_{i},\theta_{ij})_{i,j\in I,i\geq j}$ from $\bf{C}$ to $\textbf{C}\mathrm{-}\textbf{Rep}\Gamma_I$.

 %And we
%still use $H_n^{\Gamma}$, defined similarly as in definition 4.8, to denote the functor from $\textbf{C}$-$\textbf{Rep}\Gamma$ to
%$\textbf{Rep}\Gamma$. And it is easy to know that all $H_n^{\Gamma},S,S^{\Gamma},\mathop{lim^{\Gamma}}\limits^{\bullet},lim^{\Gamma}$ are functors.

\begin{Definition}
Let $\Gamma$ be a quiver. Define a functor $S^{\Gamma}:\textbf{Top}\mathrm{-}\textbf{Rep}\Gamma\rightarrow\textbf{C}\mathrm{-}\textbf{Rep}\Gamma$. For all $(T,f)\in\textbf{Top}\mathrm{-}\textbf{Rep}\Gamma$, $S^{\Gamma}(T,f)=(\mathop{A}\limits^{\bullet},h)$ where $\mathop{A}\limits^{\bullet}(i)=S(T(i))$ for all $i\in\Gamma$
and $h_{ij}=S(f_{ij})$ for all $i\rightarrow j$ in $\Gamma$. For any $\alpha:(T,f)\rightarrow(T^{'},f^{'})$, $S^{\Gamma}(\alpha):
S^{\Gamma}(T,f)\rightarrow S^{\Gamma}(T^{'},f^{'})$ with $(S^{\Gamma}(\alpha))_{i}=S(\alpha_i)$.

Similarly, for each $n\geq-1$, we define a functor $S^{\Gamma}_n:\textbf{Top}\mathrm{-}\textbf{Rep}\Gamma\rightarrow \textbf{Ab}\mathrm{-}\textbf{Rep}\Gamma$. For all $(T,f)\in\textbf{Top}\mathrm{-}\textbf{Rep}\Gamma$,$S^{\Gamma}_{n}(T,f)=(G,g)$
where $G(i)=S_n(T(i))$ for all $i\in\Gamma$ and $g_{ij}=S_n(f_{ij})$ for all $i\rightarrow j$ in $\Gamma$. For any $\alpha:(T,f)\rightarrow(T^{'},f^{'})$, $S^{\Gamma}_n(\alpha):
S^{\Gamma}_n(T,f)\rightarrow S^{\Gamma}_n(T^{'},f^{'})$ with $(S^{\Gamma}_n(\alpha))_{i}=S_n(\alpha_i)$.
\end{Definition}

\begin{Example}
Let $\Gamma$ be a quiver: $\cdot\longrightarrow\cdot\longleftarrow\cdot$.
For any $(T,f): X_1\longrightarrow X_2\longleftarrow X_3$ with maps $f_{12}, f_{32} \in$ $\textbf{Top}\mathrm{-}\textbf{Rep}\Gamma$,  $$S^{\Gamma}(T,f)= S(X_1)\longrightarrow S(X_2)\longleftarrow S(X_3)$$ with maps $S(f_{12}), S(f_{32})$ and
$S^{\Gamma}_n(T,f)= S_n(X_1)\longrightarrow S_n(X_2)\longleftarrow S_n(X_3)$ with maps $S_n(f_{12}), S_n(f_{32})$.
\end{Example}

\begin{Definition}
Let $\Gamma$ be a quiver. For any $(T,f)\in\textbf{Top}\mathrm{-}\textbf{Rep}\Gamma$, define $\mathop{lim^{\Gamma}}\limits^{\bullet}(T,f)=\mathop{lim^{\Gamma}}\limits^{\bullet}
(S^{\Gamma}(T,f))$.
\end{Definition}

From the definition, we know that $(\mathop{lim^{\Gamma}}\limits^{\bullet}(T,f))_n=lim^{\Gamma}(S_n^{\Gamma}(T,f))$ for all $n\geq-1$.

\begin{Corollary}
Let $\Gamma$ be a quiver. Then $\mathop{lim^{\Gamma}}\limits^{\bullet}:\textbf{Top}\mathrm{-}\textbf{Rep}\Gamma\rightarrow \textbf{C}$ is a functor.
\end{Corollary}

\begin{Example}
Let $\Gamma$ be a finite quiver $\cdot\rightarrow\cdot\rightarrow\cdots\rightarrow\cdot$. For any $(T,f)\in\textbf{Top}\mathrm{-}\textbf{Rep}\Gamma$,
assume $(T,f)=X_1\rightarrow X_2\rightarrow\dots\rightarrow X_n$. Obviously, $\mathop{lim^{\Gamma}}\limits^{\bullet}(T,f)\simeq S(X_1)$.
\end{Example}
\begin{Lemma}\label{lem4.9}
Let $\Gamma$ be a connected quiver, and $(T,f)\in\textbf{Top}\mathrm{-}\textbf{Rep}\Gamma$ with $f_{ij}$ being isomorphisms for
all $i\rightarrow j$ in $\Gamma$ and $f_{\alpha}=f_{\beta}$ if both $\alpha$ and $\beta$ are arrows from $i$ to $j$. Then $\mathop{lim^{\Gamma}}\limits^{\bullet}(T,f)$ is isomorphic to a subcomplex of $S(T(i))$ where $i\in\Gamma$. Furthermore, $\mathop{lim^{\Gamma}}\limits^{\bullet}(T,f)\simeq S(T(i))$ for $i\in\Gamma$ if $\Gamma$ is acyclic.
\end{Lemma}
\begin{proof}
Since $\Gamma$ is connected and $f_{ij}$ are isomorphisms for all $i\rightarrow j$ in $\Gamma$, we have $T(i)\simeq T(j)$ and $S(T(i))\simeq S(T(j))$ for all $i,j\in\Gamma$. For each $j\in\Gamma$, define $L_{j}:\mathop{lim^{\Gamma}}\limits^{\bullet}(T,f)\rightarrow S(T(i))$ such that for all $n\geq0$ $(L_{j})_n:lim^{\Gamma}(S^{\Gamma}_n(T,f))\rightarrow S_n(T(i)),(\alpha_{i})_{i\in\Gamma}\mapsto
\alpha_j$. We will show $(L_{j})_n$ are injective for all $j\in\Gamma$ and $n\geq0$. If $\alpha_i=\beta_i$ for any $(\alpha_i)_{i\in\Gamma},(\beta_i)_{i\in\Gamma}\in lim^{\Gamma}(S^{\Gamma}_n(T,f))$. Since $\Gamma$ is connected and $f_{ij}$ are isomorphisms
for all $i\rightarrow j$ in $\Gamma$, there exists an isomorphism $g_t:S_n(T(i))\rightarrow S_n(T(t))$ for each $t\in\Gamma$
such that $\alpha_t=g_t(\alpha_i)=g_t(\beta_i)=\beta_t$. Therefore, $(\alpha_i)_{i\in\Gamma}=(\beta_{i})_{i\in\Gamma}$
and $(L_{j})_n$ is injective. This implies $L_{j}:\mathop{lim^{\Gamma}}\limits^{\bullet}(T,f)\rightarrow S(T(j))$ is injective for each $j\in\Gamma$.
There exist isomorphisms $g_{j_1,j_2}$ from $S(T(j_1))$ to $S(T(j_2))$ because $\Gamma$ is connected and $f_{ij}$ are isomorphisms for all $i\rightarrow j$ in $\Gamma$. We have the following diagram:
$$\xymatrix{
&\mathop{lim^{\Gamma}}\limits^{\bullet}(T,f)\ar[r]^{L_{j_1}}\ar[d]^{L_{j_2}}&S(T(j_1))\ar[dl]^{g_{j_1,j_2}}\\
&S(T(j_2))
}$$
If $\Gamma$ contains no circle, obviously $L_{j}$ is also surjective, that is $L_{j}$ is isomorphic, for each $j\in\Gamma$. Then,
$\mathop{lim^{\Gamma}}\limits^{\bullet}(T,f)\simeq S(T(j))$ for all $j\in\Gamma$.

\end{proof}

\begin{Definition}\label{Def4.9}
Let $\Gamma$ be a quiver, and $(T,f)\in\textbf{Top}\mathrm{-}\textbf{Rep}\Gamma$. Define
$H_n(T,f)=H_n(\mathop{lim^{\Gamma}}\limits^{\bullet}(T,f))$ for all $n\geq0$.
\end{Definition}

\begin{Example}\label{Ex4.11}
Let $\Gamma$ be a connected quiver, $(T,f)\in\textbf{Top}\mathrm{-}\textbf{Rep}\Gamma$ with $T(i)=\{x_i\}$ for all $i\in\Gamma$.
Obviously, $\mathop{lim^{\Gamma}}\limits^{\bullet}(T,f)\simeq S(T(i))$. Thus, $H_n(T,f)\simeq H_n(T(i))$, and then $H_n(T,f)=0$ $ \forall n\geq1,H_0(T,f)=\mathds{Z}$.
\end{Example}

As in Definition \ref{Def4.9}, $H_n(T,f)=H_n(\mathop{lim^{\Gamma}}\limits^{\bullet}(T,f))$ for all $n\geq0$. For any
morphism $(h_i)_{i\in\Gamma}:(T,f)\rightarrow (T^{'},f^{'})$ and $n\geq0$,
$H_n((h_i)_{i\in\Gamma})=H_n(\mathop{lim^{\Gamma}}\limits^{\bullet}(f)):H_n(T,f)\rightarrow H_n(T^{'},f^{'})$.
For each $n\geq0$, $H_n(\mathrm{-})$ is a functor from $\textbf{Top}\mathrm{-}\textbf{Rep}\Gamma$ to $\textbf{Ab}$.

Then, by Lemma \ref{lem4.9}, we can obtain the following corollaries.

\begin{Corollary}
For each $n\geq0$, $H_n(\mathrm{-})$ is a functor from $\textbf{Top}\mathrm{-}\textbf{Rep}\Gamma$ to $\textbf{Ab}$.
\end{Corollary}
\begin{Corollary}\label{Cor4.13}
Let $\Gamma$ be a connected and acyclic quiver, and $(T,f)\in\textbf{Top}\mathrm{-}\textbf{Rep}\Gamma$ with
all $f_{ij}$ being isomorphisms and $f_{\alpha}=f_{\beta}$ if both $\alpha$ and $\beta$ are arrows from $i$ to $j$. Then $H_n(T,f)\simeq H_n(T(i)),\forall i\in\Gamma$ and $n\geq0$.
\end{Corollary}

\begin{Definition}\label{Def4.13}
Define functors $H_n^{\Gamma}$ from $\textbf{Top}\mathrm{-}\textbf{Rep}\Gamma$ to $\textbf{Ab}\mathrm{-}\textbf{Rep}\Gamma$, $n\geq1$. For any object $(T,f)$ in
$\textbf{Top}\mathrm{-}\textbf{Rep}\Gamma$, $H_n^{\Gamma}(T,f)$=$(H_n,g_n)$ such that $H_n(i)=H_n(X_i)$ and $(g_n)_{ij}=H_n(f_{ij})$. For any
morphism $(h_i)_{i\in\Gamma}:(T,f)\rightarrow (T^{'},f^{'})$ in $\textbf{Top}\mathrm{-}\textbf{Rep}\Gamma$, $H_n^{\Gamma}((h_i)_{i\in\Gamma})=(H_n(h_i))_{i\in\Gamma}$.
\end{Definition}

Similarly, for each $n\in\mathds{Z}$ we can define a functor from $\textbf{C}\mathrm{-}\textbf{Rep}\Gamma$ to $\textbf{Ab}\mathds{-}\textbf{Rep}\Gamma$, also denoted as $H_n^{\Gamma}(\mathrm{-})$.
\begin{Example}
Let $\Gamma$ be the quiver:
$$\xymatrix{ &\cdot\ar[r]&\cdot&\cdot\ar[l]
}
$$
For any $(T,f): X_1\longrightarrow X_2\longleftarrow X_3$ with maps $f_{12}, f_{32} \in$ $\textbf{Top}\mathrm{-}\textbf{Rep}\Gamma$, $$H_n^{\Gamma}(T,f)= H_n(X_1)\longrightarrow H_n(X_2)\longleftarrow H_n(X_3)$$ with maps $H_n(f_{12}), H_n(f_{32})$.
\end{Example}
\begin{Example}
For any $(\mathop{C}\limits^{\bullet},f):\mathop{C_1}\limits^{\bullet}\longrightarrow\mathop{C_2}\limits^{\bullet}\longleftarrow\mathop{C_3}\limits^{\bullet}$ with chain morphisms $f_{12},f_{32}$ in $\textbf{C}\mathrm{-}\textbf{Rep}$, $H_n^{\Gamma}(\mathop{C}\limits^{\bullet},f):H_n(\mathop{C_1}\limits^{\bullet})\longrightarrow H_n(\mathop{C_2}\limits^{\bullet})\longleftarrow H_n(\mathop{C_3}\limits^{\bullet})$ with abelian group morphisms $H_n(f_{12}),H_n(f_{32})$.
\end{Example}

Let $\Gamma$ be a quiver. We already know $H_n\mathop{lim^{\Gamma}}\limits^{\bullet},lim^{\Gamma}H_n^{\Gamma}$ are two functors from $\textbf{C}\mathrm{-}\textbf{Rep}\Gamma$
to $\textbf{Ab}$ for each $n\geq0$.

\begin{Theorem}\label{Cor4.17} (\textbf{Commutativity of $H_n$ and $\mathop{lim^{\Gamma}}\limits^{\bullet}$)}
Let $\Gamma$ be a quiver. For each $n\geq0$,
there is a natural transformation $\rho_n$ from functor $H_n\mathop{lim^{\Gamma}}\limits^{\bullet}$ to functor $lim^{\Gamma}H_n^{\Gamma}$.

\end{Theorem}
\begin{proof}
We already know that $H_n\mathop{lim^{\Gamma}}\limits^{\bullet},lim^{\Gamma}H_n^{\Gamma}$ are functors from $\textbf{C}\mathrm{-}\textbf{Rep}\Gamma$
to $\textbf{Ab}$. For any $(\mathop{A}\limits^{\bullet},f)\in\textbf{C}\mathrm{-}\textbf{Rep}\Gamma$, define $\rho_{n,(\mathop{A}\limits^{\bullet},f)}:H_n(\mathop{lim^{\Gamma}}\limits^{\bullet}(\mathop{A}\limits^{\bullet},f))\rightarrow lim^{\Gamma}(H_n^{\Gamma}(\mathop{A}\limits^{\bullet},f)),
cls((a_i)_{i\in\Gamma})\mapsto (cls(a_i))_{i\in\Gamma}$. It is easy to show that $\rho_{n,(\mathop{A}\limits^{\bullet},f)}$
is a morphism between groups. And for any $\alpha:(\mathop{A}\limits^{\bullet},f)\rightarrow (\mathop{B}\limits^{\bullet},g)$, we have $\rho_{n,(\mathop{B}\limits^{\bullet},g)}H_n(\mathop{lim^{\Gamma}}\limits^{\bullet}
(\alpha))=lim^{\Gamma}(H_n^{\Gamma}(\alpha))\rho_{n,(\mathop{A}\limits^{\bullet},f)}$.
\end{proof}
Let $\Gamma$ be a quiver. We already know $H_n,lim^{\Gamma}H_n^{\Gamma}$ are two functors from $\textbf{Top}\mathrm{-}\textbf{Rep}\Gamma$
to $\textbf{Ab}$ for each $n\geq0$.
\begin{Corollary}
There exists a natural transformation $\psi:H_n\rightarrow lim^{\Gamma}H_n^{\Gamma}$ for each $n\geq0$.
\end{Corollary}
\begin{proof}
From the definition, we know that $H_n=H_n\mathop{lim^{\Gamma}}\limits^{\bullet}S^{\Gamma},lim^{\Gamma}H^{\Gamma}_n=lim^{\Gamma}H^{\Gamma}_nS^{\Gamma}:\textbf{Top}\mathrm{-}\textbf{Rep}\Gamma
\rightarrow\textbf{Ab}$. Applying Theorem \ref{Cor4.17}, we can get the inclusion.
\end{proof}

\begin{Lemma}\label{Lem4.19}
Let $\Gamma$ be a quiver, $\Gamma=\Gamma^{'}\cup \Gamma^{''}$ where $\Gamma^{'}$ and $\Gamma^{''}$ are sub-quivers of $\Gamma$ and $\Gamma^{'}\cap\Gamma^{''}=\emptyset$, $(G,g)$ be a  $\Gamma$-representation via abelian groups, $(\mathop{C}\limits^{\bullet},h)$ be a  $\Gamma$-representation via
chain complexes of abelian groups. Denote $(G_1,g_1)$ as a  $\Gamma^{'}$-representation via abelian groups such that $G_1(i)=G(i), (g_1)_{ij}=g_{ij}$ for all $i\in\Gamma^{'}_{0},i\rightarrow j\in\Gamma^{'}_1$. Similarly, define $(G_2,g_2),(\mathop{C_1}\limits^{\bullet},h_1),(\mathop{C_2}\limits^{\bullet},h_2)$. Then
$lim^{\Gamma}(G,g)\simeq lim^{\Gamma^{'}}(G_1,g_1)\prod lim^{\Gamma^{''}}(G_2,g_2)$ and
$\mathop{lim^{\Gamma}}\limits^{\bullet}(\mathop{C}\limits^{\bullet},h)\simeq\mathop{lim^{\Gamma^{'}}}\limits^{\bullet}(\mathop{C_1}\limits^{\bullet},h_1)\prod\mathop{lim^{\Gamma^{''}}}\limits^{\bullet}(\mathop{C_2}\limits^{\bullet},h_2)$

\end{Lemma}

\begin{proof}
Define $\phi$ from $lim^{\Gamma}(G,g)$ to $lim^{\Gamma^{'}}(G_1,g_1)\prod lim^{\Gamma^{''}}(G_2,g_2)$,
$(x_i)_{i\in\Gamma}\mapsto ((x_i)_{i\in\Gamma^{'}},(x_i)_{i\in\Gamma^{''}})$. Then it is easy to verify $\phi$ is an abelian group isomorphism. Similarly, we can prove $\mathop{lim^{\Gamma}}\limits^{\bullet}(\mathop{C}\limits^{\bullet},h)\simeq\mathop{lim^{\Gamma^{'}}}\limits^{\bullet}(\mathop{C_1}\limits^{\bullet},h_1)\prod\mathop{lim^{\Gamma^{''}}}\limits^{\bullet}(\mathop{C_2}\limits^{\bullet},h_2)$.
\end{proof}

\begin{Theorem}\label{Lem4.20}
Let $\Gamma$ be a quiver, $\Gamma=\Gamma^{'}\cup \Gamma^{''}$ where $\Gamma^{'}$ and $\Gamma^{''}$ are sub-quivers of $\Gamma$ and $\Gamma^{'}\cap\Gamma^{''}=\emptyset$, and $(T,f)\in$ $\textbf{Top}\mathrm{-}\textbf{Rep}\Gamma$. Denote $(T_1,f_1)$ as a top-representation of $\Gamma^{'}$ such that
$T_1(i)=T(i), (f_1)_{ij}=f_{ij}$ for any $i, j \in\Gamma^{'}_{0}$. Similarly, denote $(T_2,f_2)$ as a top-representation of $\Gamma^{''}$. Then
$H_n(T,f)\simeq H_n(T_1,f_1)\prod H_n(T_2,f_2)$ for all $n\geq0$.
\end{Theorem}
\begin{proof}
Since $\Gamma=\Gamma^{'}\cup \Gamma^{''}, \Gamma^{'}\cap\Gamma^{''}=\emptyset$, we have
\begin{equation}
\begin{split} \mathop{lim^{\Gamma}}\limits^{\bullet}(T,f)=&\mathop{lim^{\Gamma}}\limits^{\bullet}(S^{\Gamma}(T,f))\\ \simeq
&\mathop{lim^{\Gamma^{'}}}\limits^{\bullet}(S^{\Gamma^{'}}(T_1,f_1))\prod \mathop{lim^{\Gamma^{''}}}\limits^{\bullet}(S^{\Gamma^{''}}(T_2,f_2))\\=&\mathop{lim^{\Gamma^{'}}}\limits^{\bullet}(T_1,f_1)
\prod \mathop{lim^{\Gamma^{''}}}\limits^{\bullet}(T_2,f_2)
\end{split}
\end{equation}
 according to Lemma \ref{Lem4.19}. Thus, $$H_n(T,f)=H_n(\mathop{lim^{\Gamma}}\limits^{\bullet}(T,f))\simeq H_n(\mathop{lim^{\Gamma^{'}}}\limits^{\bullet}(T_1,f_1))
\prod H_n(\mathop{lim^{\Gamma^{''}}}\limits^{\bullet}(T_2,f_2))=H_n(T_1,f_1)\prod H_n(T_2,f_2)$$ for all $n\geq0$
\end{proof}
\begin{Corollary}\label{Cor4.21}
Let $\Gamma$ be a quiver and $(\Gamma_j)_{1\leq j\leq m}$ be all components of $\Gamma$, $(T,f)$ be a top-representation of $\Gamma$. Then
$H_n(T,f)=\prod_{1\leq j\leq m}H_n(T_j,f_j)$ for all $n\geq0$ where $(T_j,f_j)$ defined as above.
\end{Corollary}
\begin{proof}
Note $\Gamma_{j_1}\cap\Gamma_{j_2}=\emptyset$ for $j_1\neq j_2$ and $\Gamma=\cup_{1\leq j\leq m}\Gamma_j$. Then the conclusion can be proved by induction on $m$.
\end{proof}
\begin{Corollary}
Let $\Gamma$ be an acyclic quiver having $m$ components, X be any topological space. Donate $(T,f)$ as a top-representation of $\Gamma$ where
$T(i)=X$ and $f_{ij}=Id_{X}$. Then $H_n(T,f)=(H_n(X))^m$ for all $n\geq0$
\end{Corollary}
\begin{proof}
Using Corollary \ref{Cor4.13} and \ref{Cor4.21}.
\end{proof}

 Let
$S^n=\{(x_1,\dots,x_{n+1})\in\mathds{R}^{n+1}|x_1^{2}+\dots,+x_{n+1}^2=1\}$ denote the $n$-dimensional {\bf unit sphere}.

For $n=1$, define a continuous map $F:S^{1}\rightarrow S^{1}$ with $(cos\theta,sin\theta)\mapsto(cos(\theta+t),sin(\theta+t))$ where $ut\neq2\pi$ for all
$u\in\mathds{Z}$. Apparently, $F$ is a homeomorphism and $F^{k}(x)\neq x$ for all $x\in S^{1},k\geq1$.

\begin{Theorem}\label{The4.22}
Let $\Gamma$ be a connected quiver, and $(T,f)$ be a top-representation with $T(i)=S^{1},f_{ij}=F$. Then $H_1(T,f)=\mathds{Z}$ if $\Gamma$
is acyclic, and $H_1(T,f)=0$ if $\Gamma$ contains oriented circle.
\end{Theorem}
\begin{proof}
If $\Gamma$ is acyclic, then $H_1(T,f)=H_1(S^{1})$ according to Corollary \ref{Cor4.13}.

 If $\Gamma$ contains circle. Assume, the lengthen of this circle is $m$($m\geq1$) and this circle is $i\rightarrow i+1\rightarrow i+2\rightarrow\dots\rightarrow i+m-1\rightarrow i$. Then we will show $\mathop{lim^{\Gamma}}\limits^{\bullet}(T,f)=0$. For all $(\alpha_j)_{j\in\Gamma}\in(\mathop{lim^{\Gamma}}\limits^{\bullet}(T,f))_n$ where $n\geq0$, we have
 \begin{equation}
 \begin{split}
 (f_{i,i+1})_{*}(\alpha_i)&=F_{*}(\alpha_i)=\alpha_{i+1},\\
 (f_{i+1,i+2})_{*}(\alpha_{i+1})&=F_{*}(\alpha_{i+1})=\alpha_{i+2},\\
 \dots &\dots, \\
 (f_{i+m-2,i+m-1})_{*}(\alpha_{i+m-2})&=F_{*}(\alpha_{i+m+2})=\alpha_{m-1}, \\ (f_{i+m-1,i})_{*}(\alpha_{i+m-1})&=F_{*}(\alpha_{i+m-1})=\alpha_i.
 \end{split}
\end{equation}
Thus, we have
$(F_{*})^{m}(\alpha_{i})=(F^{m})^{*}(\alpha_i)=\alpha_{i}.$

If $\alpha_{i}\neq0$, assume, $\alpha_i=k_1f_1+k_2f_2+\dots+k_vf_v$ where $v\in\mathds{N}^{+}, k_i\neq0, f_{i}\neq
f_j$ if $i\neq j$, then we have
\begin{equation}
\begin{split}
 (F^m)^{*}(\alpha_i)=&k_1F^{m}f_1+\dots+k_vF^{m}f_v\\
 =&k_1f_1+\dots+k_vf_v.
 \end{split}
 \end{equation}
Since $f_1,\dots,f_v,F^mf_1,\dots,F^{m}f_v$
are basic elements in free abelian group $S_n(S^1)$, then $\{f_1,\dots,f_v\}=\{F^mf_1,\dots,F^mf_v\}$, that is, $F^mf_i=f_{s(i)}$ for all $1\leq i\leq v$ where $s\in S_{m}$. Then $(F^m)^{m!}f_i=F^{mm!}f_i=f_{s^{m!}(i)}=f_i$. Thus for all $x\in S^1$ $f_i(x)=F^{mm!}(f_i(x))$, which is a contradiction.

Now, we have $\alpha_i=0$. Since $\Gamma$ is connected and all $(f_{ij})_{*}=F_{*}$ are isomorphisms, $\alpha_j=0$ for all $j\in\Gamma$. Thus, $\mathop{lim^{\Gamma}}\limits^{\bullet}(T,f)=0$ and $H_1(T,f)=0$
\end{proof}
\begin{Remark}\label{Rem4.24}
Define a continuous map $\overline{f}:S^{1}\rightarrow S^{1},(x_1,x_2)\mapsto (-x_1,-x_2)$. Let $X$ a topological space,
$(g_i)_{1\leq i\leq n}$ be a family of distinct continuous maps from $X$ to $S^1$. We will show $\{g_1,\dots,g_n\}=\{h_1,\overline{f}h_1,\dots,
h_m,\overline{f}h_m\}$ where $n=2m,\{h_1,\dots,h_m\}\subset\{g_1,\dots,g_n\}$ if $\{g_1,\dots,g_n\}=\{\overline{f}g_1,\dots,\overline{f}g_n\}$. First, we have $\overline{f}g_i=g_{\rho(i)}$
where $\rho\in S_n$. Then $g_{i}=\overline{f}^2g_i=\overline{f}g_{\rho(i)}=g_{\rho^2(i)},\forall 1\leq i\leq n$, and thus
$\rho^2(i)=i,\forall 1\leq i\leq n$, that is, $\rho^2=(1)\in S_n$. Since $\overline{f}g_i\neq g_i,\forall 1\leq i\leq n$
, then $\rho(i)\neq i,\forall 1\leq i\leq n$. Thus $n$ is even(suppose $n=2m$) and $\rho=(i_1i_2)(i_3i_4)\dots(i_{n-1}i_n)
\in S_n$ where $\{i_1,i_2,\dots,i_n\}=\{1,2,\dots,n\}$. Let $h_t=g_{i_{2t-1}},\forall 1\leq t\leq m$. Then, $\overline{f}h_t=
\overline{f}g_{i_{2t-1}}=g_{\rho(i_{2t-1})}=g_{i_{2t}},\forall 1\leq t\leq m$. Then,$\{g_1,\dots,g_n\}=
\{g_{i_1},\dots,g_{i_n}\}=\{h_1,\overline{f}h_1,\dots,h_m,\overline{f}h_m\}$
\end{Remark}

Let $\Gamma$ be a connected quiver containing oriented circles, and $(T,f)\in\textbf{Top}\mathrm{-}\textbf{Rep}\Gamma$ with $T(i)=S^1,f_{ij}=\overline{f},\forall
i,i\rightarrow j\in\Gamma$. We select one point $i_0\in\Gamma$, from Lemma \ref{lem4.9}, we know that there
exists an injective chain map $L$ from $\mathop{lim^{\Gamma}}\limits^{\bullet}(T,f)$ to $S(S^1)$. Then we have the following theorem:

\begin{Theorem}\label{The4.24}
$H_0(\frac{S(S^1)}{ImL})\simeq\mathds{Z}_2$ if $\Gamma$ contains at least one oriented circle whose length is odd; otherwise, $H_0(\frac{S(S^1)}{ImL})=0$.
\end{Theorem}
\begin{proof}
If $\Gamma$ contains no oriented circle whose length is odd, that is, $\Gamma$ only contains oriented circles whose lengths are even.
Since $\overline{f}^{m}=Id_{S^1}$ for any positive even number $m$, then $ImL=S(S^1)$ according to Lemma \ref{lem4.9}.
Therefore, $\frac{S(S^1)}{ImL}=0$ and $H_0(\frac{S(S^1)}{ImL})=0$ .

If $\Gamma$ contains at least one oriented circle whose length is odd.
For any positive odd number $m$, we have $\overline{f}^m=\overline{f}$.
For all $n\geq0$, let $B_n$ denote the free subgroup of $S_n(S^1)$ with basis
\begin{equation}
\{g+\overline{f}g|g\mathrm{\;is\;the\;continuous\;map
\;from\;}\Delta^n\mathrm{\;to\;}S^1\}.
\end{equation}

We first show $(ImL)_n=B_n$ for all $n\geq0$

For all $(\alpha_i)_{i\in\Gamma}\in(\mathop{lim^{\Gamma}}\limits^{\bullet}(T,f))_{n},n\geq0$, then $(f_{ij})_{*}(\alpha_i)=\overline{f}_{*}(\alpha_i)=\alpha_j$.
Since $\Gamma$ is connected and $(\overline{f}_{*})^{2}=Id$, then $\alpha_i=\alpha_{i_0}\mathrm{\;or\;}\overline{f}_{*}(\alpha_{i_0})$
for all $i\in\Gamma$. Since $\Gamma$ contains at least one oriented circle whose length $m$ is odd, then $(\overline{f}_{*})^{m}(\alpha_{i_0})=
\alpha_{i_0}$ or $(\overline{f}_{*})^{m}(\overline{f}_{*}(\alpha_{i_0}))=\overline{f}_{*}(\alpha_{i_0})$.
In both cases, we have $\overline{f}_{*}(\alpha_{i_0})=\alpha_{i_0}$.
If $\alpha_{i_0}=0$, then $L_n(\alpha_{i})_{i\in\Gamma}=\alpha_{i_0}=0\in B_n$
Assume $\alpha_{i_0}\neq0$. Then $\alpha_{i_0}=k_1(g_{11}+\dots+g_{1i_1})+\dots+k_s(g_{s1}+\dots+g_{si_s})$,
where $s$ is a positive integer, $k_i(1\leq i\leq s)$ are distinct nonzero integers, and $g_{tj}(1\leq t\leq s,1\leq j\leq i_t)$
are distinct continuous maps from $\Delta^n$ to $S^1$.
Then,
\begin{equation}
\begin{split}
\overline{f}_{*}(\alpha_{i_0})&=k_1(\overline{f}g_{11}+\dots+\overline{f}g_{1i_1})+\dots+k_s(\overline{f}g_{s1}+\dots+\overline{f}g_{si_s})\\
&=k_1(g_{11}+\dots+g_{1i_1})+\dots+k_s(g_{s1}+\dots+g_{si_s}).
\end{split}
\end{equation}
Since $k_i(1\leq i\leq s)$ are distinct nonzero integers, and $g_{tj}(1\leq t\leq s,1\leq j\leq i_t)$
are distinct continuous maps from $\Delta^n$ to $S^1$, then $\overline{f}g_{t1}+\dots+\overline{f}g_{ti_t}=g_{t1}+\dots+g_{ti_t}$.
for $1\leq t\leq s$.
Therefore, $\{\overline{f}g_{t1},\dots,\overline{f}g_{ti_t}\}=\{g_{t1},\dots,g_{ti_t}\}$ for $1\leq t\leq s$.
According to Remark \ref{Rem4.24}, there exists a subset $\{h_{t1},\dots,h_{tj_t}\}$ of $\{g_{t1},\dots,g_{ti_t}\}$
where $i_t=2j_t$ such that $\{h_{t1},\overline{f}h_{t1},\dots,h_{tj_t},\overline{f}h_{tj_t}\}=\{g_{t1},\dots,g_{ti_t}\}$.
Then $h_{t1}+\overline{f}h_{t1}+\dots+h_{tj_t}+\overline{f}h_{tj_t}=g_{t1}+\dots+g_{ti_t}$. Therefore
\begin{equation}
\begin{split}
\alpha_{i_0}&=k_1(g_{11}+\dots+g_{1i_1})+\dots+k_s(g_{s1}+\dots+g_{si_s})\\
&=k_1(h_{11}+\overline{f}h_{11}+\dots+h_{1j_1}+\overline{f}h_{1j_1})
+\dots+k_s(h_{s1}+\overline{f}h_{s1}+\dots+h_{sj_s}+\overline{f}h_{sj_s})
\end{split}.
\end{equation}
Then we have $L_n(\alpha_{i})_{i\in\Gamma}=\alpha_{i_0}\in B_n$. From the above discussion, we know $(ImL)_n\subset B_n,n\geq0$.

For any continuous map $g$ from $\Delta^n$ to $S^1$, let $\alpha_{i_0}=g+\overline{f}g$.
Then $(\overline{f}_{*})^{m}(\alpha_{i_0})=\alpha_{i_0}$ for any  positive integer $m$.
Thus, there exists $(\alpha_{i})_{i\in\Gamma}\in(\mathop{lim^{\Gamma}}\limits^{\bullet}(T,f))_n$ such that
$L_n((\alpha_i)_{i\in\Gamma})=\alpha_{i_0}=g+\overline{f}g\in(ImL)_n$.
It follows that $\{g+\overline{f}g|g\mathrm{\;is\;the\;continuous\;map
\;from\;}\Delta^n\mathrm{\;to\;}S^1\}\subset(ImL)_n$ and $B_n\subset(ImL)_n,n\geq0$.
Then we have $B_n=(ImL)_{n},n\geq0$.

Particularly, $(ImL)_0$ has basis $\{x+\overline{f}(x)|x\in S^1\}$ and $(ImL)_1$
has basis $\{\sigma+\overline{f}\sigma|\sigma\mathrm{\;is\;the\;path\;of\;}S^1\}$.
Consider the subcomplex $ImL$ of $S(S^1)$.
Since $S^1$ is path connected, it is easy to verify
\begin{equation}
Im\partial_1=\{\sum_{1\leq t\leq n}k_t(x_t+\overline{f}(x_t))|n\in\mathds{N}^{+},x_t\in S^{1},\sum_{1\leq t\leq n_i}k_t=0\}
\end{equation}
Since $(ImL)_{-1}=0$, we have $$H_0(T,f)=H_0(\mathop{lim^{\Gamma}}\limits^{\bullet}(T,f))\simeq H_0(ImL)=\frac{B_0}{Im\partial_1}.$$
Define $\theta:B_0\rightarrow\mathds{Z},\sum_{1\leq t\leq n_i}k_t(x_t+\overline{f}(x_t))\mapsto \sum_{1\leq t\leq n_i}k_t$.
It is easy to verify that $\theta$ is surjective and $Ker\theta=Im\partial_1$.
Then $\overline{\theta}:H_0(ImL)=\frac{B_0}{Im\partial_1}=\frac{B_0}{Ker\theta}\rightarrow\mathds{Z},cls(\sum_{1\leq t\leq n_i}k_t(x_t+\overline{f}(x_t)))\mapsto \sum_{1\leq t\leq n_i}k_t$ is an isomorphism.
Hence $H_0(T,f)\simeq H_0(ImL)=\frac{B_0}{Im\partial_1}=\frac{B_0}{Ker\theta}\simeq\mathds{Z}$.
Let $(\beta_i)_{i\in\Gamma}\in(\mathop{lim^{\Gamma}}\limits^{\bullet}(T,f))_{0}$ with $\beta_{i_0}=x_0+\overline{f}(x_0),x_0\in S^1$.
Since $H_0(L)(cls(\beta_i)_{i\in\Gamma})=cls(x_0+\overline{f}(x_0)),\theta(cls(x_0+\overline{f}(x_0)))=1$,
$cls((\beta_i)_{i\in\Gamma})$ is a generator of $H_0(T,f)$.

We have a short exact sequence:
$$\xymatrix{&0\ar[r]&\mathop{lim^{\Gamma}}\limits^{\bullet}(T,f)\ar[r]^{L}&S(S^1)\ar[r]^{\pi}&\frac{S(S^1)}{ImL}\ar[r]&0
}$$

Since $(\mathop{lim^{\Gamma}}\limits^{\bullet}(T,f))_{-1}=S_{-1}(S^1)=(\frac{S(S^1)}{ImL})_{-1}=0$, there exists
an exact sequence:
$$\xymatrix{&H_0(T,f)\ar[r]^{H_0(L)}&H_0(S^1)\ar[r]^{H_0(\pi)}&H_0(\frac{S(S^1)}{ImL})\ar[r]&0
}$$

And $H_0(T,f)\simeq H_0(S^1)\simeq \mathds{Z}$,$cls((\beta_i)_{i\in\Gamma})$ is a generator
of $H_0(T,f)$, for any $x,y\in S^{1}$, $cls(x)=cls(y)$ is a generator of $H_0(S^1)$.
$H_0(L)(cls((\beta_i)_{i\in\Gamma}))=cls(L((\beta_i)_{i\in\Gamma}))=
cls(\beta_{i_0})=cls(x_0+\overline{f}(x_0))=cls(x_0)+cls(\overline{f}(x_0))=2cls(x_0)$.
Thus, $ImH_0(L)=2\mathds{Z}cls(x_0)$ and $H_0(S^1)=\mathds{Z}cls(x_0)$
Then, $H_0(\frac{S(S^1)}{ImL})\simeq
\frac{H_0(S^1)}{KerH_0(\pi)}=\frac{H_0(S^1)}{ImH_0(L)}\simeq \frac{\mathds{Z}}{2\mathds{Z}}=\mathds{Z}_2$.

\end{proof}

\section{Homotopy equivalence in top-representations}

\begin{Definition}\label{Def4.26}
Let $\Gamma$ be a quiver, $(\mathop{A}\limits^{\bullet},f),(\mathop{B}\limits^{\bullet},g)\in \textbf{C}\mathrm{-}\textbf{Rep}\Gamma$ and $\alpha,\beta:(\mathop{A}\limits^{\bullet},f)\rightarrow (\mathop{B}\limits^{\bullet},g)$. We say $\alpha$ is $\textbf{homotopic}$ to $\beta$, denoted as $\alpha\simeq^t\beta$, if there exists a homotopy $F_i:\alpha_i\simeq^t\beta_i$ for each $i\in\Gamma$, and the following diagram commutes for all $n\in\mathds{Z}$ if $f_{ij}$ exists:

$$\xymatrix{
&A(i)_{n-1}\ar[r]^{(f_{ij})_{n-1}}\ar[d]_{(F_i)_{n-1}}&A(j)_{n-1}\ar[d]^{(F_j)_{n-1}}\\
&B(i)_{n}\ar[r]^{(g_{ij})_n}&B(j)_n
}.$$
\end{Definition}
We also call $F:\alpha\simeq^t\beta$ a homotopy. It is easy to show that such $``\simeq^t"$ is an equivalence in $Hom_{\textbf{C}\mathrm{-}\textbf{Rep}\Gamma}((\mathop{A}\limits^{\bullet},f),(\mathop{B}\limits^{\bullet},g))$. And for any $(\mathop{A}\limits^{\bullet},f),(\mathop{B}\limits^{\bullet},g),(\mathop{C}\limits^{\bullet},h)\in\textbf{C}\mathrm{-}\textbf{Rep}\Gamma$, $$\gamma\alpha\simeq^t\delta\beta:(\mathop{A}\limits^{\bullet},f)\rightarrow(\mathop{C}\limits^{\bullet},h)$$
if $\gamma\simeq^t\delta:(\mathop{B}\limits^{\bullet},g)\rightarrow(\mathop{C}\limits^{\bullet},h)$ and $\alpha\simeq^t\beta:(\mathop{A}\limits^{\bullet},f)\rightarrow(\mathop{B}\limits^{\bullet},g)$.

Define $K((\mathop{A}\limits^{\bullet},f),(\mathop{B}\limits^{\bullet},g))=$$\{\alpha\in Hom_{\textbf{C}\mathrm{-}\textbf{Rep}\Gamma}((\mathop{A}\limits^{\bullet},f),(\mathop{B}\limits^{\bullet},g))|\alpha\simeq^t0\}$. Then $K((\mathop{A}\limits^{\bullet},f),(\mathop{B}\limits^{\bullet},g))$ is a subgroup of $Hom_{\textbf{C}\mathrm{-}\textbf{Rep}\Gamma}((\mathop{A}\limits^{\bullet},f),(\mathop{B}\limits^{\bullet},g))$. Define the category
$\textbf{K}\mathrm{-}\textbf{Rep}\Gamma$ consisting of the same objects as $\textbf{C}\mathrm{-}\textbf{Rep}\Gamma$ with

\begin{equation}
Hom_{\textbf{K}\mathrm{-}\textbf{Rep}\Gamma}((\mathop{A}\limits^{\bullet},f),(\mathop{B}\limits^{\bullet},g))=
Hom_{\textbf{C}\mathrm{-}\textbf{Rep}\Gamma}((\mathop{A}\limits^{\bullet},f),(\mathop{B}\limits^{\bullet},g))/K((\mathop{A}\limits^{\bullet},f),(\mathop{B}\limits^{\bullet},g)).
\end{equation}
\begin{Lemma}\label{Lem4.27}
The functor $\mathop{lim^{\Gamma}}\limits^{\bullet}:\textbf{C}\mathrm{-}\textbf{Rep}\Gamma\rightarrow\textbf{C}$ induces a natural
functor $\mathop{lim^{\Gamma}}\limits^{\bullet}:\textbf{K}\mathrm{-}\textbf{Rep}\Gamma\rightarrow\textbf{K}$.
\end{Lemma}
\begin{proof}
It suffices to show that for any $\alpha\simeq^t\beta:(\mathop{A}\limits^{\bullet},f)\rightarrow(\mathop{B}\limits^{\bullet},g)$ where $(\mathop{A}\limits^{\bullet},f),(\mathop{B}\limits^{\bullet},g)\in\textbf{C}\mathrm{-}\textbf{Rep}\Gamma$, we have
$\mathop{lim^{\Gamma}}\limits^{\bullet}(\alpha)\simeq^t\mathop{lim^{\Gamma}}\limits^{\bullet}(\beta):\mathop{lim^{\Gamma}}\limits^{\bullet}(\mathop{A}\limits^{\bullet},f)\rightarrow\mathop{lim^{\Gamma}}\limits^{\bullet}(\mathop{B}\limits^{\bullet},g)$.
Assume $F:\alpha\simeq^t\beta$. Then for all $i\in\Gamma$, $F_i:\alpha_i\simeq^t\beta_i$, and we have the commutative diagram in Definition \ref{Def4.26}.

Furthermore, we have $\prod_{i\in\Gamma}F_i:\prod_{i\in\Gamma}\alpha_i\simeq^t\prod_{i\in\Gamma}\beta_i:\prod_{i\in\Gamma}\mathop{A}\limits^{\bullet}(i)\rightarrow\prod_{i\in\Gamma}\mathop{B}\limits^{\bullet}(i)$.
Look at the following diagram:
$$\xymatrix{
&\ar[r]&\prod_{i\in\Gamma}A(i)_{n+1}\ar[r]\ar[d]&\prod_{i\in\Gamma}A(i)_n\ar[r]\ar[d]\ar[dl]^{\prod_{i\in\Gamma}(F_i)_n}
&\prod_{i\in\Gamma}A(i)_{n-1}\ar[r]\ar[d]\ar[dl]^{\prod_{i\in\Gamma}(F_i)_{n-1}}&\\
&\ar[r]&\prod_{i\in\Gamma}B(i)_{n+1}\ar[r]&\prod_{i\in\Gamma}B(i)_n\ar[r]&\prod_{i\in\Gamma}B(i)_{n-1}\ar[r]&
}$$
We know that $\mathop{lim^{\Gamma}}\limits^{\bullet}(\mathop{A}\limits^{\bullet},f)$ and $\mathop{lim^{\Gamma}}\limits^{\bullet}(\mathop{B}\limits^{\bullet},g)$ are subcomplexes of $\prod_{i\in\Gamma}\mathop{A}\limits^{\bullet}(i)$  and
$\prod_{i\in\Gamma}\mathop{B}\limits^{\bullet}(i)$ respectively. Thus $\mathop{lim^{\Gamma}}\limits^{\bullet}(\alpha)\simeq^t\mathop{lim^{\Gamma}}\limits^{\bullet}(\beta)$ holds if
$\prod_{i\in\Gamma}(F_i)_n((\mathop{lim^{\Gamma}}\limits^{\bullet}(\mathop{A}\limits^{\bullet},f))_n)\subset(\mathop{lim^{\Gamma}}\limits^{\bullet}(\mathop{B}\limits^{\bullet},g))_{n+1}$.
For all $\prod_{i\in\Gamma}(a_i)\in(\mathop{lim^{\Gamma}}\limits^{\bullet}(\mathop{A}\limits^{\bullet},f))_n$, $\prod_{i\in\Gamma}(F_i)_n(\prod_{i\in\Gamma}(a_i))=
\prod_{i\in\Gamma}(F_i)_n(a_i)$, if $g_{ij}$ exists, according to the communnative diagram in Definition \ref{Def4.26}, we have
\begin{equation}
(g_{ij})_{n+1}((F_i)_n(a_i))=(F_j)_n((f_{ij})_n(a_i))=(F_j)_n(a_j).
\end{equation}
This completes the proof.
\end{proof}
\begin{Definition}\label{Def4.28}
Let $\Gamma$ be a quiver, $(T,f),(T^{'},f^{'})\in\textbf{Top}\mathrm{-}\textbf{Rep}\Gamma$ and $\mu,\nu:(T,f)\rightarrow(T^{'},f^{'})$.
We say $\mu$ is $\textbf{homotopic}$ to $\nu$, denoted as $\mu\simeq^t\nu$, if there exists a homotomy $H_i:\mu_i\simeq^t\nu_i$ for each $i\in\Gamma$, and the following diagram commutes if $f_{ij}$ exists:
$$\xymatrix{
&T(i)\times I\ar[r]^{H_i}\ar[d]_{f_{ij}\times1}&T^{'}(i)\ar[d]^{f_{ij}^{'}}\\
&T(j)\times I\ar[r]^{H_j}&T^{'}(j)
}$$.
\end{Definition}
Meantime, we also call $H:\mu\simeq^t\nu$ a homotopy.

Remark that $H:\mu\simeq^t\nu$  in the category $\textbf{Top}\mathrm{-}\textbf{Rep}\Gamma$ has the same properties as the above
$F:\mu\simeq^t\nu$ in the category $\textbf{C}\mathrm{-}\textbf{Rep}\Gamma$.

\begin{Remark}(\cite{AT})\label{Rem4.29}
We recall some well-known conclusions in algebraic topology. Let $X$ be any topological space, and define $\lambda_{i}^{X}
:X\rightarrow X\times I, x\mapsto(x,i)$. Then there exist $P_n^{X}:S_n(X)\rightarrow S_{n+1}(X\times I)$ such that $P^{X}:
S(\lambda_0^{X})\simeq^t S(\lambda_1^{X}):S(X)\rightarrow S(X\times I)$.
For any continuous map $f:X\rightarrow Y$, there is a commutative diagram:
$$\xymatrix{
&S_n(X)\ar[r]^{P_n^X}\ar[d]_{S_n(f)}&S_{n+1}(X\times I)\ar[d]^{S_{n+1}(f\times1)}\\
&S_n(Y)\ar[r]^{P_n^Y}&S_{n+1}(Y\times I)
}$$.
\end{Remark}

From Definition \ref{Def4.28}, we claim that $H^{\Gamma}_n(\mu)=H^{\Gamma}_n(\nu)$ for all $n\geq0$ if $\mu\simeq^{t}\nu:(T,f)\rightarrow(T^{'},f^{'})$. First, let us look at the parallel result for
topological spaces.

Comparing with the fact that $H_n(f)=H_n(g)$ for all $n\geq0$ if $f,g:X\rightarrow Y$ are homotopic  (\cite{AT}, p75),
we have the following theorem.

\begin{Theorem}\label{The5.5}({\bf Homotopy Axiom})
Let $(T,f),(T^{'},f^{'})\in\textbf{Top}\mathrm{-}\textbf{Rep}\Gamma$, and $\mu\simeq^{t}\nu:(T,f)\rightarrow(T^{'},f^{'})$. Then
$S^{\Gamma}(\mu)\simeq^{t} S^{\Gamma}(\nu):S^{\Gamma}(T,f)\rightarrow S^{\Gamma}(T^{'},f^{'})$, and therefore $H_n(\mu)=H_n(\nu)$ for all $n\geq0$.

\end{Theorem}

\begin{proof}
Assume $H:\mu\simeq^{t}\nu$. Then $H_i:\mu_i\simeq^{t}\nu_i$, that is, $H_i:T(i)\times I\rightarrow T^{'}(i)$ and $H_i(-,0)=\mu_i,H_i(-,1)=\nu_i$.
According to Remark \ref{Rem4.29}, there exist $P_n^i:S_n(T(i))\rightarrow S_{n+1}(T(i)\times I)$ such that $P^i:S(\lambda^i_0)\simeq^{t} S(\lambda^i_1$),
where $\lambda^i_k:T(i)\rightarrow T(i)\times I, x\mapsto (x,k)$ for $k=0,1$.
Then
\begin{equation}\label{eq9}
S(H_i)P^i:S(H_i)S(\lambda^i_0)=S(\mu_i)\simeq^{t} S(H_i)S(\lambda^i_1)=S(\nu_i).
\end{equation}
We have the following commutative diagram if $f_{ij}$ exists:
\begin{equation}\label{eq10}
\xymatrix{
&S_n(T(i))\ar[r]^{P_n^i}\ar[d]^{S_n(f_{ij})}&S_{n+1}(T(i)\times I)\ar[r]^{S_{n+1}(H_i)}\ar[d]^{S(f_{ij}\times1)}&S_{n+1}(T^{'}(i))\ar[d]^{S_{n+1}(f^{'}_{ij})}\\
&S_n(T(j))\ar[r]^{P^{j}_n}&S_{n+1}(T(j)\times I)\ar[r]^{S_{n+1}(H_j)}&S_{n+1}(T^{'}(j))
}
\end{equation}
The commutativity of the left square is based on Remark \ref{Rem4.29}, and the commutativity of the right square is based Definition \ref{Def4.28}.
Therefore one gets $S^{\Gamma}(\mu)\simeq^{t} S^{\Gamma}(\nu):S^{\Gamma}(T,f)\rightarrow S^{\Gamma}(T^{'},f^{'})$ based on equation (\ref{eq9}) and commutative diagram (\ref{eq10}). According to Lemma \ref{Lem4.27}, $\mathop{lim^{\Gamma}}\limits^{\bullet}(S^{\Gamma}(\mu))\simeq^{t} \mathop{lim^{\Gamma}}\limits^{\bullet}(S^{\Gamma}(\nu)):\mathop{lim^{\Gamma}}\limits^{\bullet}(S^{\Gamma}(T,f))\rightarrow \mathop{lim^{\Gamma}}\limits^{\bullet}(S^{\Gamma}(T^{'},f^{'}))$.
Thus for all $n\geq0$, we have: $$H_n(\mu)=H_n(\mathop{lim^{\Gamma}}\limits^{\bullet}(S^{\Gamma}(\mu)))=H_n(\mathop{lim^{\Gamma}}\limits^{\bullet}(S^{\Gamma}(\nu)))=H_n(\nu).$$ .
\end{proof}

\begin{Definition}
Let $\Gamma$ be a quiver, and $(T,f),(T^{'},f^{'})\in\textbf{Top}\mathrm{-}\textbf{Rep}\Gamma$. We say $(T,f)$ and $(T^{'},f^{'})$
have the $\textbf{same homotopy type}$%,(in this case, we also call $(T,f)$ and $(T^{'},f^{'})$ are $\textbf{homotopy equivalence}$)
, if there exists $\mu:(T,f)\rightarrow(T^{'},f{'})$ and $\nu:(T^{'},f^{'})\rightarrow(T,f)$ such that
$\mu\nu\simeq^t Id_{(T^{'},f^{'})}$ and $\nu\mu\simeq^t Id_{(T,f)}$.
\end{Definition}

For any two topological spaces $X,Y$, we already know their homotopy groups are isomorphic if they have the same homotopy type. Then it is  natural to consider the parallel problem for two top-representations having the same
homotopy type.
\begin{Corollary}
 (\cite{AT}, p79) If $X$ and $Y$ have the same homotopy type, then $H_n(X)\simeq H_n(Y)$ for all $n\geq0$, where the isomorphism is induced by any homotopy equivalence.
\end{Corollary}
\begin{Corollary}
If $(T,f),(T^{'},f^{'})\in\textbf{Top}\mathrm{-}\textbf{Rep}\Gamma$ have the same homotopy type, then $H_n(T,f)\simeq H_n(T^{'},f^{'})$ for all $n\geq0$.
\end{Corollary}
\begin{Corollary}
Let $(T,f)$ be a top-representation of a quiver $\Gamma$, with all $T(i)$ being convex linear topological spaces and all $f_{ij}$ being affine maps. If there exist $x_i\in T(i)$ for all $i\in\Gamma$ such that $f_{ij}(x_i)=x_j$, then $H_0(T,f)=\mathds{Z}$ and $H_n(T,f)=0$ for all
$n\geq1$.
\end{Corollary}
\begin{proof}
Let $(T^{'},f^{'})$ be a top-representation of $\Gamma$ with $T^{'}(i)=\{x_i\}$ for all $i\in\Gamma$. Define $\alpha:(T,f)\rightarrow(T^{'},f^{'})$ and $\beta:(T^{'},f^{'})\rightarrow(T,f)$,
 such that $\alpha_{i}:T(i)\rightarrow T^{'}(i)$, $x\mapsto x_i$ and $\beta_{i}:T^{'}(i)\rightarrow T(i)$, $x_i\mapsto x_i$.
 Then we have $\beta\alpha=1_{(T^{'},f^{i})}$
 and define $H_i:T(i)\times I\rightarrow T(i),(x,t)\mapsto tx+(1-t)x_i$. Since $f_{ij}$ are affine maps, we have the following diagram:
 $$\xymatrix{
&T(i)\times I\ar[r]^{H_i}\ar[d]_{f_{ij}\times1}&T(i)\ar[d]^{f_{ij}}\\
&T(j)\times I\ar[r]^{H_j}&T(j).
}$$
Thus $H:\alpha\beta\simeq^t 1_{(T,f)}$ and $H_n(T,f)\simeq H_n(T^{'},f^{'})$ for all $n\geq0$. This completes the proof according to Example
\ref{Ex4.11}.
\end{proof}

Let $(T,f)$ be a top-representation of a quiver $\Gamma$, and $T^{'}(i)$ is a subspace of $T(i)$ for each $i\in\Gamma$ such that $f_{ij}(T^{'}(i))\subset T^{'}(j)$. Then $(T^{'},f)$ is a top-subrepresentation of top-representation $(T,f)$ and $S^{\Gamma}(T^{'},f)\subset S^{\Gamma}(T,f)$.
Thus we can define $H_n((T,f),(T^{'},f))=H_n(\frac{\mathop{lim^{\Gamma}}\limits^{\bullet}(T,f)}{\mathop{lim^{\Gamma}}\limits^{\bullet}(T^{'},f)})$ (referring to the original correspondence in \cite{AL}). For any two top-representations $(T,f),(T^{'},f^{'})$ of the quiver $\Gamma$ and $(T^{''},f),(T^{'''},f^{'})$ being top-subrepresentations of $(T,f),(T^{'},f^{'})$ respectively and $\alpha:((T,f),(T^{''},f))\rightarrow((T^{'},f^{'}),(T^{'''},f^{'}))$ meaning the restriction of $\alpha$ on $(T^{''},f)$ is a morphism from
$(T^{''},f)$ to $(T^{'''},f^{'})$, then $\alpha$ induces morphisms $H_n(\alpha):H_n((T,f),(T^{''},f))\rightarrow H_n((T^{'},f^{'}),(T^{'''},f^{'}))$ for all $n\geq0$.

\begin{Theorem}
Let $\Gamma$ be a quiver, and $(T,f)$ s top-representation of $\Gamma$. And there exist $x_i\in T(i)$ for all $i\in\Gamma$ such that
$f_{ij}(x_i)=x_j$ if $f_{ij}:T(i)\rightarrow T(j)$ exists. Then, $H_n((T,f),(X_0,f))\simeq H_n(T,f)$ for all $n\geq1$ and
there exists a short exact sequence: $0\longrightarrow\mathds{Z}\longrightarrow H_0(T,f)\longrightarrow H_0((T,f),(X_0,f))\longrightarrow0$, where $(X_0,f)$ is a top-subrepresentation of $(T,f)$ with $X_0(i)=\{x_i\}$ for all $i\in\Gamma$.
\end{Theorem}
\begin{proof}
We already know that $H_n(X_0,f)=0$ for all $n\geq1$ and $H_0(X_0,f)=\mathds{Z}$. We have a short exact sequence: $0\longrightarrow
\mathop{lim^{\Gamma}}\limits^{\bullet}(X_0,f)\longrightarrow\mathop{lim^{\Gamma}}\limits^{\bullet}(T,f)\longrightarrow\frac{\mathop{lim^{\Gamma}}\limits^{\bullet}(T,f)}{\mathop{lim^{\Gamma}}\limits^{\bullet}(X_0,f)}
\longrightarrow0$. Apply $H_n(-)$ to this short exact sequence, we have a long exact sequence:

$0\longrightarrow H_0(X_0,f)\longrightarrow
H_0(T,f)\longrightarrow H_0((T,f),(X_0,f))\longrightarrow H_1(X_0,f)\longrightarrow H_1(T,f)\longrightarrow H_1((T,f),(X_0,f))\longrightarrow\dots$. \\Since $H_n(X_0,f)=0$ for all $n\geq1$ and $H_0(X_0,f)=\mathds{Z}$, we have $H_n((T,f),(X_0,f))\simeq H_n(T,f)$ for all $n\geq1$ and a short exact sequence: $0\longrightarrow\mathds{Z}\longrightarrow H_0(T,f)\longrightarrow H_0((T,f),(X_0,f))\longrightarrow0$.
\end{proof}
\begin{Remark}(\cite{AT})  Let $X$ be a topological space, and $X=X_1^{\circ}\cup X_2^{\circ}$ where $X_1$ and $X_2$ are subspaces of $X$. Then $H_n(S(X)/(S(X_1)+S(X_2)))=0$ for all $n\geq0$.
\end{Remark}

\begin{Theorem}\label{thm5.12}({\bf Excision theorem})
Let $(T,f)$ be a top-representation of a quiver $\Gamma$ such that $T(i)=(T^{'}(i))^{\circ}\cup(T^{''}(i))^{\circ}$ and $f_{ij}(T^{'}(i))\subset T^{'}(j),f_{ij}(T^{''}(i))\subset T^{''}(j)$. Then for $n=0,1$, it holds that $$H_n(\mathop{lim^{\Gamma}}\limits^{\bullet}(\frac{S^{\Gamma}(T,f)}{S^{\Gamma}(T^{'},f)+S^{\Gamma}(T^{''},f)}))=0.$$
\end{Theorem}
\begin{proof}
We first prove that the functor $lim^{\Gamma}: \mathbf{Ab\mathrm{-}Rep\Gamma}\rightarrow\mathbf{Ab}$ is left exact for any quiver $\Gamma$.
For any short exact sequence in $\mathbf{Ab\mathrm{-}Rep\Gamma}$
$$\xymatrix{
&0\ar[r]&(A,f)\ar[r]^{\alpha}&(B,g)\ar[r]^{\beta}&(C,h)\ar[r]&0
},$$
we want show $0\longrightarrow lim^{\Gamma}(A,f)\longrightarrow lim^{\Gamma}(B,g)\longrightarrow lim^{\Gamma}(C,h)$ is exact.
Apparently, $lim^{\Gamma}(\alpha)$ is injective and $lim^{\Gamma}(\beta)lim^{\Gamma}(\alpha)=0$. It suffices to show $kerlim^{\Gamma}(\beta)\subset Imlim^{\Gamma}(\alpha)$.For all $(b_i)_{i\in\Gamma}\in Kerlim^{\Gamma}(\beta)$, we have $g_{ij}(b_i)=b_j$
and $\beta_i(b_i)=0$. From the above exact sequence, we have $a_i\in A_i$ for each $i\in\Gamma$ such that $\alpha_i(a_i)=b_i$.
Then $b_j=\alpha_j(a_j)=g_{ij}(b_i)=g_{ij}\alpha_i(a_i)=\alpha_j f_{ij}(a_i)$, and $a_j=f_{ij}(a_j)$ since $\alpha$ is injective.
Therefore, $(a_i)_{i\in\Gamma}\in lim^{\Gamma}(A,f)$ and $(b_i)_{i\in\Gamma}=lim^{\Gamma}((a_i)_{i\in\Gamma})\in Imlim^{\Gamma}(\alpha)$

Let $(\mathop{C}\limits^{\bullet},g)=\frac{S^{\Gamma}(T,f)}{S^{\Gamma}(T^{'},f)+S^{\Gamma}(T^{''},f)}$, then, $\mathop{C}\limits^{\bullet}(i)=\frac{S((T(i)))}{S(T^{'}(i))+S(T^{''}(i))}$. From Remark, we know that $\mathop{C}\limits^{\bullet}(i)$ is exact. And let $(C^n,f^n)$ be in
$\mathbf{Ab\mathrm{-}Rep\Gamma}$ with $C^n(i)=(\mathop{C}\limits^{\bullet}(i))_n$ and $(f^n)_{ij}=(f_{ij})_n$. From the above discussion, we know that there is an exact sequence in $\mathbf{Ab\mathrm{-}Rep\Gamma}:0\longrightarrow(C^0,f^0)\longrightarrow(C^1,f^1)\longrightarrow(C^2,f^2)\longrightarrow\dots.$ Since
$lim^{\Gamma}$ is a left exact functor, we have an exact sequence:$0\longrightarrow lim^{\Gamma}(C^0,f^0)\longrightarrow
lim^{\Gamma}(C^1,f^1)\longrightarrow lim^{\Gamma}(C^2,f^2)$. From the definition, $(\mathop{lim^{\Gamma}}\limits^{\bullet}(\frac{S^{\Gamma}(T,f)}{S^{\Gamma}(T^{'},f)+S^{\Gamma}(T^{''},f)}))_n=lim^{\Gamma}(C^n,f^n)$ for all $n\geq0$.
Thus, $H_n(\mathop{lim^{\Gamma}}\limits^{\bullet}(\frac{S^{\Gamma}(T,f)}{S^{\Gamma}(T^{'},f)+S^{\Gamma}(T^{''},f)}))=0$ for $n=0,1$.
\end{proof}
\begin{Corollary}
Let $\Gamma$ be a quiver, then the functor $lim^{\Gamma}:\textbf{Ab}\mathrm{-}\textbf{Rep}\Gamma\rightarrow \mathbf{Ab}$ is left exact.
\end{Corollary}

\begin{Corollary}
Let $\Gamma$ be a quiver, the the functor $\mathop{lim^{\Gamma}}\limits^{\bullet}:\textbf{C}\mathrm{-}\textbf{Rep}\Gamma\rightarrow\textbf{C}$ is left exact.
\end{Corollary}

\section{On the functor $At^{\Gamma}$ from $\textbf{Top}\mathrm{-}\textbf{Rep}\Gamma$ to
$\textbf{Top}$}

Given a quiver $\Gamma$, it is  natural to associate each top-representation $(T,f)$ of $\Gamma$ with a concrete topology, which is determined
by both the quiver $\Gamma$ and the top-representation $(T,f)$. First, we define a functor $At^{\Gamma}$ from $\textbf{Top}\mathrm{-}\textbf{Rep}\Gamma$ to
$\textbf{Top}$, which preserves homotopy equivalence property. Then by analyzing topological space $At^{\Gamma}(T,f)$, we can obtain some properties of quivers, such as the connectivity of quivers and the numbers of components of  quivers. Last, we established the relationship between the homotopy groups of top-representation
and the homotopy groups of the corresponding topological space.

Throughout this section, we always assume the quiver $\Gamma$ is finite.

\begin{Remark}\label{Rem5.1}
let $X$ be a topological space and $D$ a subset of $X\times X$, we know that $D$ can generate an equivalence relationship $"\simeq^D"$ on $X$. Let $f:X\rightarrow Y$ be a continuous map between two topological spaces, and $D,E$ be subsets of $X\times X,Y\times Y$ respectively. If $(f\times f)(D)\subset E$, that is, $\{(f(x_1),f(x_2))|(x_1,x_2)\in D\}\subset E$, then $f$ induces a continuous map $\overline{f}:X/\simeq^D\rightarrow Y/\simeq^E$, $\overline{x}\mapsto\overline{f(x)}$. That is because:

Let $C^{-1}=\{(x_1,x_2)|(x_2,x_1)\in C\}, K^{X}=\{(x,x)|x\in X\}$, define $\widetilde{C}=C\cup C^{-1}\cup K^{X}$. And define
$\overline{C}=\{(x_1,x_2)|(x_1,x_2)\in\widetilde{ C}\mathrm{ \;or\; there\; exist\; }v_1,\dots,v_n\in X\mathrm{\;such\;that\;(x_1,v_1),(v_1,v_2),\dots,(v_n,x_2)\in\widetilde{C}}\}$. Then it is easy to show that
$C$ generates $\overline{C}$. Similarly, we define $\widetilde{D},\overline{D}$. Then it suffices to show that
$(f\times f)(\overline{C})\subset\overline{D}$.Since $(f\times f)(C)\subset D$, $(f\times f)(C^{-1})\subset
D^{-1}$ and $(f\times f)(K^{X})\subset K^{Y}$, that is, $(f\times f)(\widetilde{C})\subset\widetilde{D}$.
Then, obviously, $(f\times f)(\overline{C})\subset\overline{D}$.

\end{Remark}
\begin{Definition}\label{Def6.2}
For a quiver $\Gamma$,  $(T,f)\in\textbf{Top}\mathrm{-}\textbf{Rep}\Gamma$, denote:
$$D=\{(x_i,x_j)|x_i\in T(i);x_j\in T(j); f_{ij}(x_i)=x_j;\forall i\rightarrow j\in\Gamma\}\subset\coprod_{i\in\Gamma}T(i)\times\coprod_{i\in\Gamma}T(i).$$
 We define $At^{\Gamma}$ from $\textbf{Top}\mathrm{-}\textbf{Rep}\Gamma$ to
$\textbf{Top}$ satisfying that \\
$\bullet$  $At^{\Gamma}(T,f)=(\coprod_{i\in\Gamma}T(i))/\simeq^{D}$;\\
$\bullet$ for any morphism $\alpha: (T,f)\rightarrow (T^{'},f^{'})$, define $$At^{\Gamma}(\alpha): At^{\Gamma}(T,f)\rightarrow
At^{\Gamma}(T^{'},f^{'})\;\; \text{such that} \;\;\overline{x}\mapsto\overline{(\coprod_{i\in\Gamma}\alpha_{i})(x)}.$$
\end{Definition}

\begin{Theorem}\label{The6.3}
Let $\Gamma$ be a quiver. Then $At^{\Gamma}$ is a functor from $\textbf{Top}\mathrm{-}\textbf{Rep}\Gamma$ to
$\textbf{Top}$.
\end{Theorem}
\begin{proof}
For any two objects $(T,f),(T^{'},f^{'})\in\textbf{Top}\mathrm{-}\textbf{Rep}\Gamma$, and any morphism $\alpha:
(T,f)\rightarrow (T^{'},f^{'})$. We first show $At^{\Gamma}(\alpha):At^{\Gamma}(T,f)\rightarrow
At^{\Gamma}(T^{'},f^{'})$ is a well-defined continuous map. From definition, we have
\begin{equation}
\begin{split}
At^{\Gamma}(T,f)=&(\coprod_{i\in\Gamma}T(i))/\simeq_{D},\\
At^{\Gamma}(T^{'},f^{'})=&(\coprod_{i\in\Gamma}T^{'}(i))/\simeq_{D^{'}}
\end{split}
\end{equation}
 where
 \begin{equation}
 \begin{split}
D=&\{(x_i,x_j)|x_i\in T(i);x_j\in T(j); f_{ij}(x_i)=x_j;\forall i\rightarrow j\in\Gamma\},\\
D^{'}=&\{(x_i^{'},x_j^{'})|x_i^{'}\in T^{'}(i);x_j^{'}\in T^{'}(j); f_{ij}^{'}(x_i^{'})=x_j^{'};\forall i\rightarrow j\in\Gamma\}.
\end{split}
\end{equation}
And $\coprod_{i\in\Gamma}\alpha_i:\coprod_{i\in\Gamma}T(i)\rightarrow\coprod_{i\in\Gamma}T^{'}(i)$ is a continuous map.

 For all
$(x_i,x_j)\in D$ and $i\rightarrow j\in\Gamma$, $((\coprod_{i\in\Gamma}\alpha_i)(x_i),(\coprod_{i\in\Gamma}\alpha_i)(x_j))=(\alpha_i(x_i),\alpha_j(x_j))$, and
$f_{ij}^{'}(\alpha_{i}(x_i))=\alpha_j(f_{ij}(x_i))=\alpha_j(x_j)$. Thus, $((\coprod_{i\in\Gamma}\alpha_i)(x_i),(\coprod_{i\in\Gamma}\alpha_i)(x_j))\in D^{'}$. According to Remark \ref{Rem5.1}, $\coprod_{i\in\Gamma}\alpha_i$
induces a continuous map $At^{\Gamma}(\alpha):At^{\Gamma}(T,f)\rightarrow At^{\Gamma}(T^{'},f^{'}),\overline{x}\mapsto\overline{(\coprod_{i\in\Gamma}\alpha_{i})(x)}$, that is, for all $x_i\in T(i),i\in\Gamma, At^{\Gamma}(\alpha)(\overline{x_i})=\overline{\alpha_i(x_i)}$.

Obviously, we have
$At^{\Gamma}(\beta\alpha)=At^{\Gamma}(\beta)At^{\Gamma}(\alpha), At^{\Gamma}(1_{(T,f)})=1_{At^{\Gamma}(T,f)}$, where
 $1_{(T,f)}$ is the identity morphism of $(T,f)$, $\alpha \in Hom_{\textbf{Top}\mathrm{-}\textbf{Rep}\Gamma}((T,f),(T^{'},f^{'}))$, $\beta \in Hom_{\textbf{Top}\mathrm{-}\textbf{Rep}\Gamma}((T^{'},f^{'}),(T^{''},f^{''}))$.
\end{proof}
\begin{Theorem}\label{The6.4} ({\bf Homotopy invariant})
 Let $\Gamma$ be a quiver. For any $\alpha\simeq^t\beta:(T,f)\rightarrow(T^{'},f^{'})$ in $\textbf{Top}\mathrm{-}\textbf{Rep}\Gamma$. Then,   as both morphisms from $At^{\Gamma}(T,f)$ to $At^{\Gamma}(T^{'},f^{'})$  in $\textbf{Top}$, it holds $At^{\Gamma}(\alpha)\simeq^t At^{\Gamma}(\beta)$.
 \end{Theorem}
 \begin{proof}
 Assume that $H:\alpha\simeq^t\beta$, then $H_i:\alpha_i\simeq^t\beta_i$ for all $i\in\Gamma$( that is $H_i:T(i)\times I\rightarrow
 T^{'}(i)$, $H_i(x,0)=\alpha_i(x),H_i(x,1)=\beta_i(x)$ for all $x\in T(i)$) and $H_j(f_{ij}\times1)=f_{ij}^{'}H_i$
 for all $i\rightarrow j$ in $\Gamma$. Thus $H:(T\times I,f\times1)\rightarrow(T^{'},f^{'})$ is a morphism in $\textbf{Top}\mathrm{-}\textbf{Rep}\Gamma$ where
 $(T\times I)(i)=T(i)\times I,(f\times1)_{ij}=f_{ij}\times1:T(i)\times I\rightarrow T(j)\times I$. Hence $At^{\Gamma}(H):At^{\Gamma}(T\times I,f\times1)\rightarrow At^{\Gamma}(T^{'},f^{'})$, $\overline{(x_i,t)}\mapsto\overline{H_i(x_i,t)}$ for $(x_i,t)\in T(i)\times I,i\in\Gamma$.

 Next, we will show $At^{\Gamma}(T\times I,f\times1)\simeq At^{\Gamma}(T,f)\times I$. We have
  \begin{equation}
  \begin{split}
  At^{\Gamma}(T\times I,f\times1)=&(\coprod_{i\in\Gamma}(T(i)\times I))/\simeq^D,\\
   At^{\Gamma}(T,f)=&(\coprod_{i\in\Gamma}T(i))/\simeq^C
  \end{split}
  \end{equation}
  where
 \begin{equation}
 \begin{split}
 D=&\{((x_i,t),(x_j,t))|(x_i,t)\in T(i)\times I;(x_j,t)\in T(j)\times I;
 (f_{ij}(x_i),t)=(x_j,t),\forall i\rightarrow j\in\Gamma\}\\
 C=&\{(x_i,x_j)|x_i\in T(i);x_j\in T(j); f_{ij}(x_i)=x_j;\forall i\rightarrow j\in\Gamma\}.
 \end{split}
 \end{equation}
 Let $p$ be the natural map from $\coprod_{i\in\Gamma}T(i)$ to $(\coprod_{i\in\Gamma}T(i))/\simeq^C$, and $q$ be the natural map from $\coprod_{i\in\Gamma}(T(i)\times I)$ to $(\coprod_{i\in\Gamma}(T(i)\times I))/\simeq^D$. Then $p\times1:\coprod_{i\in\Gamma}(T(i)\times I)=(\coprod_{i\in\Gamma}T(i))\times I\rightarrow((\coprod_{i\in\Gamma}T(i))/\simeq^C)\times I;(x,t)\mapsto(\overline{x},t).$ For all $((x_i,t),(x_j,t))\in D$, we have $(f_{ij}(x_i),t)=(x_j,t)$, that is, $f_{ij}(x_i)=x_j$ and $(x_i,x_j)\in C$. Thus $(p\times1)(x_i,t)=(\overline{x_i},t)=(\overline{x_j},t)=
 (p\times1)(x_j,t)$. Then there exists a map $\overline{p\times1}:(\coprod_{i\in\Gamma}(T(i)\times I))/\simeq^D\rightarrow ((\coprod_{i\in\Gamma}T(i))/\simeq^C)\times I,\overline{(x,t)}\mapsto(\overline{x},t)$ such that the following diagram commutes:
 $$\xymatrix{
 &\coprod_{i\in\Gamma}(T(i)\times I)\ar[r]^{q}\ar[d]^{p\times1}&(\coprod_{i\in\Gamma}(T(i)\times I))/\simeq^D\ar[dl]^{\overline{p\times1}}\\
 &((\coprod_{i\in\Gamma}T(i))/\simeq^C)\times I
 }$$
 Since $q$ is an identification and $(\overline{p\times1})q=p\times1$ is continuous, it follows that $\overline{p\times1}$ is continuous.  Clearly, $\overline{p\times1}$ is surjective. For any
 $\overline{(x_1,t)},\overline{(x_2,t)}\in(\coprod_{i\in\Gamma}(T(i)\times I))/\simeq^D$, if $(x_1,x_2)\in\widetilde{C}$, then:

 $\textcircled{1}$ $x_1=x_2$, thus $\overline{(x_1,t)}=\overline{(x_2,t)}$;

 $\textcircled{2}$ $(x_1,x_2)\in C$, thus
 $((x_1,t),(x_2,t))\in D$ and $\overline{(x_1,t)}=\overline{(x_2,t)}$;

 $\textcircled{3}$ $(x_1,x_2)\in C^{-1}$, thus
 $((x_1,t),(x_2,t))\in D^{-1}$ and $\overline{(x_1,t)}=\overline{(x_2,t)}$.

 \noindent In a word, for any $\overline{(x_1,t)},\overline{(x_2,t)}\in(\coprod_{i\in\Gamma}(T(i)\times I))/\simeq^D$, $(x_1,x_2)\in\widetilde{C}$ implies
 $\overline{(x_1,t)}=\overline{(x_2,t)}$. For any $\overline{(x_1,t)},\overline{(x_2,t)}\in(\coprod_{i\in\Gamma}(T(i)\times I))/\simeq^D$, we want to show that $(x_1,x_2)\in\overline{C}$ implies $\overline{(x_1,t)}=\overline{(x_2,t)}$. Since $(x_1,x_2)\in\overline{C}$,
 there exist $v_1,\dots,v_n\in\coprod_{i\in\Gamma}T(i)$ such that $(x_1,v_1),(v_1,v_2),\dots,(v_n,x_2)\in\widetilde{C}$. Thus, $\overline{(x_1,t)}=\overline{(v_1,t)}=\dots=\overline{(v_n,t)}=\overline{(x_2,t)}$. For any $\overline{(x_1,t)},\overline{(x_2,t)}\in(\coprod_{i\in\Gamma}(T(i)\times I))/\simeq^D$, if $\overline{p\times1}(\overline{(x_1,t)})=(\overline{x_1},t)=\overline{p\times1}(\overline{(x_2,t)})=(\overline{x_2},t)$,
  then $\overline{x_1}=\overline{x_2}$, which implies $\overline{(x_1,t)}=\overline{(x_2,t)}$. Therefore $\overline{p\times1}$ is also injective. We have a commutative diagram:
  $$\xymatrix{
 &\coprod_{i\in\Gamma}(T(i)\times I)=(\coprod_{i\in\Gamma}T(i))\times I\ar[r]\ar[d]^{p\times1}&(\coprod_{i\in\Gamma}(T(i)\times I))/\simeq^D\\
 &((\coprod_{i\in\Gamma}T(i))/\simeq^C)\times I\ar[ur]_{(\overline{p\times1})^{-1}}
 }$$
 \noindent  Since $I$ is locally compact and $p$ is an identification, we have that $p\times1$ is an identification. Therefore, $(\overline{p\times1})^{-1}$ is continuous. One has $$(\overline{p\times1})^{-1}:At^{\Gamma}(T,f)\times I\simeq At^{\Gamma}(T\times I,f\times1)\;\;\; \text{via}\;\;\; (\overline{x},t)\mapsto\overline{(x,t)}.$$
 Let $$F=At^{\Gamma}(H)(\overline{p\times1})^{-1}:At^{\Gamma}(T,f)\times I\rightarrow At^{\Gamma}(T^{'},f^{'})\;\;\; \text{via}\;\;\; (\overline{x_i},t)\mapsto\overline{H_i(x_i,t)}$$ for $(x_i,t)\in T(i)\times I,i\in\Gamma$. Then $$F(\overline{x_i},0)=\overline{H_i(x_i,0)}=\overline{\alpha_i(x_i)}=At^{\Gamma}(\alpha)(\overline{x_i}),
 F(\overline{x_i},1)=\overline{H_i(x_i,1)}=\overline{\beta_i(x_i)}=At^{\Gamma}(\beta)(\overline{x_i})$$ for all $x_i\in T(i),i\in\Gamma$. And,
 $F:At^{\Gamma}(\alpha)\simeq^t At^{\Gamma}(\beta):At^{\Gamma}(t,f)\rightarrow At^{\Gamma}(T^{'},f^{'})$ in $\textbf{Top}$.

\end{proof}

\begin{Theorem}\label{The6.5}(\textbf{Connectivity invariant})
Let $\Gamma$ be a connected quiver, and $(T,f)\in\textbf{Top}\mathrm{-}\textbf{Rep}\Gamma$ with all $T(i)$ being connected. Then,
$At^{\Gamma}(T,f)$ is connected.
\end{Theorem}
\begin{proof}
Suppose on the contrary, then $At^{\Gamma}(T,f)=X\cup Y,X\cap Y=\emptyset$ and $X,Y$ are open close subsets of $At^{\Gamma}(T,f)$.
Let $p$ be the natural map from $\coprod_{i\in\Gamma}T(i)$ to $At^{\Gamma}(T,f)$.

Since $p$ is surjective,
$P^{-1}(X),p^{-1}(Y)$ are nonempty and open close subsets of $\coprod_{i\in\Gamma}T(i)$, and $\coprod_{i\in\Gamma}T(i)
=p^{-1}(X)\cup p^{-1}(Y),p^{-1}(X)\cap p^{-1}(Y)=\emptyset$.

 Since $T(i)$ are connected and $P^{-1}(X),p^{-1}(Y)$ are an open and close
subsets of $\coprod_{i\in\Gamma}T(i)$, we have $T(i)\subset P^{-1}(X)$ or $T(i)\cap P^{-1}(X)=\emptyset$,$T(i)\subset P^{-1}(Y)$ or $T(i)\cap P^{-1}(Y)=\emptyset$ for each $i\in\Gamma$.

Therefore, $p^{-1}(X)=\cup_{i\in\Gamma^{'}}T(i),p^{-1}(X)=\cup_{i\in\Gamma^{''}}T(i)$
where $point(\Gamma)=\Gamma^{'}\cup\Gamma^{''},\Gamma^{'}\cap\Gamma^{''}=\emptyset$.

Because $\Gamma$ is connected,
there exist $i_1\in\Gamma^{'},i_2\in\Gamma^{''}$ such that there is an arrow from $i_1$ to $i_2$ in $\Gamma$.
Select $x\in T(i_1)\subset p^{-1}(X)$, and then $f_{i_1,i_2}(x)\in T(i_2)\subset p^{-1}(Y)$. Thus $\overline{x}=p(x)=\overline{f_{i_1,i_2}(x)}=p(f_{i_1,i_2}(x))\in X\cap Y$. This is a contradiction, and so $At^{\Gamma}(T,f)$ is connected.

\end{proof}
\begin{Theorem}\label{The6.6}
Let $\Gamma$ be a quiver, $\Gamma=\Gamma^{'}\cup \Gamma^{''}$ where $\Gamma^{'}$ and $\Gamma^{''}$ are nontrivial sub-quivers of $\Gamma$ and $\Gamma^{'}\cap\Gamma^{''}=\emptyset$, and $(T,f)\in$ $\textbf{Top}\mathrm{-}\textbf{Rep}\Gamma$. Denote $(T^{'},f^{'})$ as a top-representation of $\Gamma^{'}$ such that
$T^{'}(i)=T(i), (f^{'})_{ij}=f_{ij}$ for any $i, j \in\Gamma^{'}$. Similarly, denote $(T^{''},f^{''})$ as a top-representation of $\Gamma^{''}$.
Then, $At^{\Gamma}(T,f)\simeq At^{\Gamma^{'}}(T^{'},f^{'})\coprod At^{\Gamma^{''}}(T^{''},f^{''})$.
\end{Theorem}
\begin{proof}
We know
\begin{equation}
\begin{split}
At^{\Gamma}(T,f)&=(\coprod_{i\in\Gamma}T(i))/\simeq^{C},C=\{(x_i,x_j)|x_i\in T(i);x_j\in T(j); f_{ij}(x_i)=x_j;\forall i\rightarrow j\in\Gamma\}\\
At^{\Gamma^{'}}(T^{'},f^{'})&=(\coprod_{i\in\Gamma^{'}}T(i))/\simeq^{C^{'}},C^{'}=\{(x_i,x_j)|x_i\in T(i);x_j\in T(j); f_{ij}(x_i)=x_j;\forall i\rightarrow j\in\Gamma^{'}\}\\
At^{\Gamma^{''}}(T^{''},f^{''})&=(\coprod_{i\in\Gamma^{''}}T(i))/\simeq^{C^{''}},C^{''}=\{(x_i,x_j)|x_i\in T(i);x_j\in T(j); f_{ij}(x_i)=x_j;\forall i\rightarrow j\in\Gamma^{''}\}.
\end{split}
\end{equation}

Since $\Gamma=\Gamma^{'}\cup \Gamma^{''}, \;\Gamma^{'}\cap\Gamma^{''}=\emptyset$, we have $C=C^{'}\cup C^{''}, \; C^{'}\cap C^{''}=\emptyset$.
Denote the natural map $p:\coprod_{i\in\Gamma}T(i)\rightarrow At^{\Gamma}(T,f),x\mapsto\overline{x}$.
Let $X_1=p(\cup_{i\in\Gamma^{'}}T(i)),X_2=p(\cup_{i\in\Gamma^{''}}T(i))$. Clearly, $At^{\Gamma}(T,f)=X_1\cup X_2$.
We will show $X_1\cap X_2=\emptyset$. For any $x_1\in\cup_{i\in\Gamma^{'}}T(i),x_2\in\cup_{i\in\Gamma^{''}}T(i)$,
$(x_1,x_2)$ is not in $\widetilde{C}$ since $\Gamma^{'}$ and $\Gamma^{''}$ are not connected.
Then for any $x_1\in\cup_{i\in\Gamma^{'}}T(i),x_2\in\cup_{i\in\Gamma^{''}}T(i)$, $(x_1,x_2)$ is not in $\overline{C}$, which implies $X_1\cap X_2=\emptyset$.

Suppose $(x_1,x_2)\in\overline{C}$. There exist $v_1,v_2,\dots,v_n\in\coprod_{i\in\Gamma}T(i)=\cup_{i\in\Gamma^{'}}T(i)\cup(\cup_{i\in\Gamma^{''}}T(i))$
such that $(x_1,v_1),(v_1,v_2),\dots,(v_n,x_2)\in\widetilde{C}$. Since $x_1\in\cup_{i\in\Gamma^{'}}T(i),
x_2\in\cup_{i\in\Gamma^{''}}T(i)$, we  have $(x,y)\in\widetilde{C}$ where $x\in\cup_{i\in\Gamma^{'}}T(i),
y\in\cup_{i\in\Gamma^{''}}T(i)$, which is a contradiction.

It is easy to show  $p^{-1}(X_1)=\cup_{i\in\Gamma^{'}}T(i),p^{-1}(X_2)=\cup_{i\in\Gamma^{''}}T(i)$.
Since $\cup_{i\in\Gamma^{'}}T(i),\cup_{i\in\Gamma^{''}}T(i)$ are open sets of $\cup_{i\in\Gamma}T(i)$,
we have that $X_1,X_2$ are open sets of $At^{\Gamma}(T,f)$ and $At^{\Gamma}(T,f)=X_1\cup X_2,X_1\cap X_2=\emptyset$.
Thus, $At^{\Gamma}(T,f)\simeq X_1\coprod X_2$.

It suffices to show $X_1\simeq At^{\Gamma^{'}}(T^{'},f^{'}),X_2\simeq At^{\Gamma^{''}}(T^{''},f^{''})$.
Define $$f:At^{\Gamma^{'}}(T^{'},f^{'})\rightarrow X_1,\;\;\; \text{via}\;\;\; \overline{x_1}\mapsto\overline{x_1};\;\;\;\text{and }\;\;\;
g:At^{\Gamma^{''}}(T^{''},f^{''})\rightarrow X_2,\;\;\; \text{via}\;\;\; \overline{x_2}\mapsto\overline{x_2}.$$ It is easy to show
$f,g$ are well-defined homeomorphisms. This completes the proof.

\end{proof}

\begin{Corollary}\label{Cor6.7}
Let $\Gamma$ be a quiver, and $(\Gamma^{k})_{1\leq k\leq n}$ be all components of $\Gamma$.
Denote $(T^{k},f^{k})$ as a top-representation of $\Gamma^{k}$ such that
$T^{k}(i)=T(i), (f^{k})_{ij}=f_{ij}$ for any $i, j \in\Gamma^{k}$.
Then, $At^{\Gamma}(T,f)\simeq\coprod_{1\leq k\leq n}At^{\Gamma^{k}}(T^{k},f^{k})$.
\end{Corollary}

\begin{Corollary}\label{Cor6.8}

Let $\Gamma$ be a quiver, and $(T,f)\in\textbf{Top}\mathrm{-}\textbf{Rep}\Gamma$ with all $T(i)$ being connected.
Then the number of components of $\Gamma$ is equal to the number of components of $At(T,f)$.
\end{Corollary}
\begin{proof}
Theorem \ref{The6.5} and Corollary \ref{Cor6.7}.
\end{proof}

\begin{Theorem}\label{The6.9}
Let $\Gamma$ be a quiver, $(T,f)\in\textbf{Top}\mathrm{-}\textbf{Rep}\Gamma$, and $(T^{'},f),(T^{''},f)$ be sub-representations
of $(T,f)$ such that $T(i)=T^{'}(i)\cup T^{''}(i),T^{'}(i)\cap T^{''}(i)=\emptyset$ for all $i\in\Gamma$. Let $L_1:(T^{'},f)\rightarrow(T,f)$
and $L_2:(T^{''},f)\rightarrow (T,f)$ be inclusions. Then,

(i)\;  $At^{\Gamma}(L_1)$ and $At^{\Gamma}(L_2)$ are injective
and $$At^{\Gamma}(T,f)=ImAt^{\Gamma}(L_1)\cup ImAt^{\Gamma}(L_2),\;\;ImAt^{\Gamma}(L_1)\cap ImAt^{\Gamma}(L_2)=\emptyset.$$

(ii)\; Furthermore, $At^{\Gamma}(T^{'},f)\simeq ImAt^{\Gamma}(L_1)$ if $T^{'}(i)$ is an open subset of $T(i)$ for each $i\in\Gamma$.
\end{Theorem}
\begin{proof}
We know
\begin{equation}
\begin{split}
At^{\Gamma}(T,f)&=(\coprod_{i\in\Gamma}T(i))/\simeq^{C},C=\{(x_i,x_j)|x_i\in T(i);x_j\in T(j); f_{ij}(x_i)=x_j;\forall i\rightarrow j\in\Gamma\}\\
At^{\Gamma}(T^{'},f)&=(\coprod_{i\in\Gamma}T^{'}(i))/\simeq^{C^{'}},C^{'}=\{(x_i,x_j)|x_i\in T^{'}(i);x_j\in T^{'}(j); f_{ij}(x_i)=x_j;\forall i\rightarrow j\in\Gamma\}\\
At^{\Gamma}(T^{''},f)&=(\coprod_{i\in\Gamma} T^{''}(i))/\simeq^{C^{''}},C^{''}=\{(x_i,x_j)|x_i\in T^{''}(i);x_j\in T^{''}(j); f_{ij}(x_i)=x_j;\forall i\rightarrow j\in\Gamma\}.
\end{split}
\end{equation}

We first show $C=C^{'}\cup C^{''},C^{'}\cap C^{''}=\emptyset;\widetilde{C}=\widetilde{C^{'}}\cup \widetilde{C^{''}},\widetilde{C^{'}}\cap \widetilde{C^{''}}=\emptyset;\overline{C}=\overline{C^{'}}\cup \overline{C^{''}},\overline{C^{'}}\cap \overline{C^{''}}=\emptyset$

Since $(\cup_{i\in\Gamma}T^{'}(i))\cap(\cup_{i\in\Gamma}T^{''}(i))=\emptyset$, $C^{'}\cap C^{''}=\emptyset$.
For any
$(x_i,x_j)\in C$, we have $x_i\in T(i)=T^{'}(i)\cup T^{''}(i),x_j\in T(j), f_{ij}(x_i)=x_j$.
If $x_i\in T^{'}(i)$, then we have $f_{ij}(x_i)=x_j\in T^{'}(j)$. Therefore, $(x_i,x_j)\in C^{'}$.
Similarly, if $x_i\in T^{''}(i)$, we have $(x_i,x_j)\in C^{''}$. Thus, $C=C^{'}\cup C^{''}$.

Since $(\cup_{i\in\Gamma}T^{'}(i))\cap(\cup_{i\in\Gamma}T^{''}(i))=\emptyset$, we have $\widetilde{C^{'}}\cap \widetilde{C^{''}}=\emptyset$.
One then has $C^{-1}=(C^{'})^{-1}\cup (C^{''})^{-1}$ due to that $C=C^{'}\cup C^{''}$. Because  $\coprod_{i\in\Gamma}T(i)=\coprod_{i\in\Gamma}T^{'}(i)\cup
\coprod_{i\in\Gamma}T^{''}(i)$, we obtain  $K=K^{'}\cup K^{''}$ where $K^{'}=\{(x,x)|x\in\coprod_{i\in\Gamma}T^{'}(i)\},K=\{(x,x)|x\in\coprod_{i\in\Gamma}T(i)\},K^{''}=\{(x,x)|x\in\coprod_{i\in\Gamma}T^{''}(i)\}$.
Thus $\widetilde{C}=C\cup C^{-1}\cup K=C^{'}\cup C^{''}\cup (C^{'})^{-1}\cup(C^{''})^{-1}\cup K^{'}\cup K^{''}
=\widetilde{C^{'}}\cup \widetilde{C^{''}},$

Since $(\cup_{i\in\Gamma}T^{'}(i))\cap(\cup_{i\in\Gamma}T^{''}(i))=\emptyset$, we have $\overline{C^{'}}\cap \overline{C^{''}}=\emptyset$.
For all $(x_1,x_2)\in\overline{C}$, there exist $v_1,v_2,\dots,v_n\in\coprod_{i\in\Gamma}T(i)$ such that $(x_1,v_1),(v_1,v_2),\dots,(v_n,x_2)\in\widetilde{C}$. If $x_1\in\coprod_{i\in\Gamma}T^{'}(i)$, then $(x_1,v_1)$ is not
in $\widetilde{C^{''}}$. Furthermore $(x_1,v_1)\in\widetilde{C}=\widetilde{C^{'}}\cup \widetilde{C^{''}}$, thus $(x_1,v_1)\in\widetilde{C^{'}}$.
We also have $v_1\in\coprod_{i\in\Gamma}T^{'}(i)$.
 Similarly, $(v_1,v_2)$ is not
in $\widetilde{C^{''}}$, and $(v_1,v_2)\in\widetilde{C}=\widetilde{C^{'}}\cup \widetilde{C^{''}}$, thus $(v_1,v_2)\in\widetilde{C^{'}}$
and $v_2\in\coprod_{i\in\Gamma}T^{'}(i)$. This process can be continued finitely, and eventually we have $(x_1,v_1),(v_1,v_2),\dots,(v_n,x_2)\in\widetilde{C^{'}}$.  It follows that $(x_1,x_2)\in\overline{C^{'}}$.
If $x_1\in\coprod_{i\in\Gamma}T^{''}(i)$, similarly, we can prove $(x_1,x_2)\in\overline{C^{''}}$. Thus $\overline{C}=\overline{C^{'}}\cup \overline{C^{''}}$.\\

{\bf The proof of (i):}

For $x\in\coprod_{i\in\Gamma}T(i)$, let $\overline{x}$ denote the equivalent class in $At^{\Gamma}(T,f)$ represented by $x$.
For $x\in\coprod_{i\in\Gamma}T^{'}(i)$, let $\overline{x}^{'}$ denote the equivalent class in $At^{\Gamma}(T^{'},f)$ represented by $x$.
Let $p$ be the natural map from $\coprod_{i\in\Gamma}T(i)$ to $At^{\Gamma}(T,f)$, and $p^{'}$ be the natural map from $\coprod_{i\in\Gamma}T^{'}(i)$ to $At^{\Gamma}(T^{'},f)$.

Look at the continuous map $At^{\Gamma}(L_1):At^{\Gamma}(T^{'},f)\rightarrow At^{\Gamma}(T,f)$.
For any $\overline{x_1}^{'},\overline{x_2}^{'}\in At^{\Gamma}(T^{'},f)$,
if $At^{\Gamma}(L_1)(\overline{x_1}^{'})=\overline{x_1}=At^{\Gamma}(L_1)(\overline{x_2}^{'})=\overline{x_2}$,
then $(x_1,x_2)\in\overline{C}=\overline{C^{'}}\cup\overline{C^{''}}$.
Since $x_1,x_2\in\coprod_{i\in\Gamma}T^{'}(i)$, we have that $(x_1,x_2)$ is not in $\overline{C^{''}}$.
Thus, $(x_1,x_2)\in\overline{C^{'}}$, that is, $\overline{x_1}^{'}=\overline{x_2}^{'}$. Therefore, $At^{\Gamma}(L_1)$ is
injective. Similarly, we can prove $At^{\Gamma}(L_2)$ is injective.

Since $\coprod_{i\in\Gamma}T(i)=\coprod_{i\in\Gamma}T^{'}(i)\cup\coprod_{i\in\Gamma}T^{''}(i)$, we obtain  $At^{\Gamma}(T,f)=ImAt^{\Gamma}(L_1)\cup ImAt^{\Gamma}(L_2)$.
We want to show $ImAt^{\Gamma}(L_1)\cap ImAt^{\Gamma}(L_2)=\emptyset$.

 Suppose on the contrary, then
there exist $x_1\in\coprod_{i\in\Gamma}T^{'}(i),x_2\in\coprod_{i\in\Gamma}T^{''}(i)$ such that
$\overline{x_1}=\overline{x_2}$, that is, $(x_1,x_2)\in\overline{C}=\overline{C^{'}}\cup\overline{C^{''}}$.
However, $x_1\in\coprod_{i\in\Gamma}T^{'}(i),x_2\in\coprod_{i\in\Gamma}T^{''}(i)$. Then we have  $(x_1,x_2)\not\in\overline{C^{'}}$
and $(x_1,x_2)\not\in\overline{C^{''}}$, which is a contradiction.

Thus, we obtain $ImAt^{\Gamma}(L_1)\cap ImAt^{\Gamma}(L_2)=\emptyset$.\\

{\bf The proof of (ii):}

If $T^{'}(i)$ is an open subset of $T(i)$ for each $i\in\Gamma$. In order to show $At^{\Gamma}(T^{'},f)\simeq ImAt^{\Gamma}(L_1)$,
it suffices to show $At^{\Gamma}(L_1):At^{\Gamma}(T^{'},f)\rightarrow ImAt^{\Gamma}(L_1)$ is an open map.
For any  open subset $\overline{X}^{'}$ of $At^{\Gamma}(T^{'},f)$, $(p^{'})^{-1}(\overline{X}^{'})$ is an open subset of
$\coprod_{i\in\Gamma}T^{'}(i)$. We will prove $p^{-1}(At^{\Gamma}(L_1)(\overline{X}^{'}))=(p^{'})^{-1}(\overline{X}^{'})$.

For all $x\in p^{-1}(At^{\Gamma}(L_1)(\overline{X}^{'}))$, we have $p(x)\in At^{\Gamma}(L_1)(\overline{X}^{'})$.
Then, there exists $\overline{x_1}^{'}\in\overline{X}^{'}$ such that $\overline{x}=At^{\Gamma}(L_1)(\overline{x_1}^{'})=\overline{x_1}$.
Thus $(x,x_1)\in\overline{C}=\overline{C^{'}}\cup \overline{C^{''}}$. $x_1\in\coprod_{i\in\Gamma}T^{'}(i)$,
and thus $(x,x_1)\in\overline{C^{'}},x\in\coprod_{i\in\Gamma}T^{'}(i)$. Hence we have $p^{'}(x)=\overline{x}^{'}=\overline{x_1}^{'}\in\overline{X}^{'}$.
Therefore, $x\in(p^{'})^{-1}(\overline{X}^{'})$ and $p^{-1}(At^{\Gamma}(L_1)(\overline{X}^{'}))\subset(p^{'})^{-1}(\overline{X}^{'})$

For all $x\in(p^{'})^{-1}(\overline{X}^{'})$,  we have $p^{'}(x)=\overline{x}^{'}\in\overline{X}^{'}$.
Thus, $At^{\Gamma}(L_1)(\overline{x}^{'})=\overline{x}=p(x)\in At^{\Gamma}(L_1)(\overline{X}^{'})$.
It implies that $x\in p^{-1}(At^{\Gamma}(L_1)(\overline{X}^{'}))$ and $(p^{'})^{-1}(\overline{X}^{'})\subset p^{-1}(At^{\Gamma}(L_1)(\overline{X}^{'}))$

Thus $p^{-1}(At^{\Gamma}(L_1)(\overline{X}^{'}))=(p^{'})^{-1}(\overline{X}^{'})$ is an open subset of
$\coprod_{i\in\Gamma}T^{'}(i)$. Since $\coprod_{i\in\Gamma}T^{'}(i)$ is an open subset of
$\coprod_{i\in\Gamma}T(i)$, it follows that $p^{-1}(At^{\Gamma}(L_1)(\overline{X}^{'}))$ is an open subset
of $\coprod_{i\in\Gamma}T(i)$. Then $At^{\Gamma}(L_1)(\overline{X}^{'})$ is an open subset
of $At^{\Gamma}(T,f)$. This implies that $At^{\Gamma}(L_1)(\overline{X}^{'})=At^{\Gamma}(L_1)(\overline{X}^{'})\cap ImAt^{\Gamma}(L_1)$
is an open subset of $ImAt^{\Gamma}(T,f)$, which means $At^{\Gamma}(L_1):At^{\Gamma}(T^{'},f)\rightarrow ImAt^{\Gamma}(L_1)$ is an open map.
This completes the proof.

\end{proof}

Let $\Gamma$ be a quiver. We have already known that $SAt^{\Gamma}, \mathop{lim^{\Gamma}}\limits^{\bullet}$ are two functors from $\textbf{Top}\mathrm{-}\textbf{Rep}\Gamma$
to $\textbf{C}$

\begin{Theorem}\label{The6.10}
Let $\Gamma$ be a quiver with finite components. There exists a natural transformation $\sigma:\mathop{lim^{\Gamma}}\limits^{\bullet}\rightarrow SAt^{\Gamma}$.
\end{Theorem}
\begin{proof}
According to Corollary \ref{Cor6.7} and the proof of Theorem \ref{Lem4.20}, we can assume the quiver is connected.
For each $(T,f)\in\textbf{Top}\mathrm{-}\textbf{Rep}\Gamma$, we try to define a chain morphism
\begin{equation}
 \sigma_{(T,f)}:\mathop{lim^{\Gamma}}\limits^{\bullet}(T,f)
\rightarrow S(At^{\Gamma}(T,f))
\end{equation}
. From the definitions, we know that
\begin{equation}
(\mathop{lim^{\Gamma}}\limits^{\bullet}(T,f))_n=lim^{\Gamma}(S^{\Gamma}_n(T,f)),
(S(At^{\Gamma}(T,f)))_n=S_n(At^{\Gamma}(T,f)).
\end{equation}
Denote the natural map $p:\coprod_{i\in\Gamma}T(i)\rightarrow At^{\Gamma}(T,f),x\mapsto\overline{x}$.
For each $i\in\Gamma$, let $p_i$ denote the composition of the inclusion, from $T(i)$ to $\coprod_{i\in\Gamma}T(i)$, and $p$.

Let $X$ be any topological space, and $F_1,F_2$ be continuous maps from $X$ to $T(i)$, $F_3$ be a continuous map from
$X$ to $T(j)$(There exists an arrow from $i$ to $j$ in $\Gamma$). We first show that $p_iF_1=p_iF_2=p_jF_3$
if $f_{ij}F_1=f_{ij}F_2=F_3$. For all $x\in X$, $f_{ij}(F_1(x))=f_{ij}(F_2(x))=F_3(x)$. Thus, we have
$p_i(F_1(x))=p_i(F_2(x))=p_j(F_3(x))$.

For any $(\alpha_i)_{i\in\Gamma}\in lim^{\Gamma}(S^{\Gamma}_n(T,f))$. If there exists an arrow from $i$ to $j$,
we know $(f_{ij})_{*}(\alpha_i)=\alpha_j$. We will show $(p_i)_{*}(\alpha_i)=(p_j)_{*}(\alpha_j)$
($p_i:T(i)\rightarrow At^{\Gamma}(T,f)$, and thus $(p_{i})_{*}:S(T(i))\rightarrow S(At^{\Gamma}(T,f))$. We will
use $\overline{\alpha_i}$ to denote $(p_i)_{*}(\alpha_i)$).

$\uppercase\expandafter{\romannumeral1}$: $\alpha_j=0$. If $\alpha_i=0$, it is trivial $\overline{\alpha_i}=\overline{\alpha_j}$.

If $\alpha_i\neq 0$, suppose $\alpha_i=k_1g_1+\dots+k_ng_n$ where $n\in\mathds{N}^{+},k_t\in\mathds{Z}-\{0\}$ and $g_u\neq g_v$ if $u\neq v$.%%%%%%%%%
Since $(f_{ij})_{*}(\alpha_i)=k_1f_{ij}g_1+\dots+k_nf_{ij}g_n=0$, we must have $\alpha_i=l_1(g_{u_1}-g_{v_1})+\dots+l_s(g_{u_s}-g_{v_s})$
where $f_{ij}g_{u_t}=f_{ij}g_{v_t},1\leq t\leq s$. From the above discussion, we know $\overline{g_{u_i}}=\overline{g_{v_i}},1\leq t\leq s$.
Thus
\begin{equation}
\overline{\alpha_i}=l_1(\overline{g_{u_1}}-\overline{g_{v_1}})+\dots+l_s(\overline{g_{u_s}}-\overline{g_{v_s}})=0=\overline{\alpha_j}
\end{equation}

$\uppercase\expandafter{\romannumeral2}$: $\alpha_j\neq0$.
Suppose
\begin{equation}
\alpha_i=k_1g_1+\dots+k_ng_n,\alpha_j=c_1h_1+\dots+c_mh_m,
\end{equation}
where $n,m\in\mathds{N}^{+},k_t,h_s\in\mathds{Z}-\{0\}$ and $g_u\neq g_v$ if $u\neq v$,$h_a\neq h_b$ if $a\neq b$.
Then
$(f_{ij})_{*}(\alpha_i)=k_1f_{ij}g_1+\dots+k_nf_{ij}g_n$. Although $g_u\neq g_v,u\neq v$, it is possible
 $f_{ij}g_u=f_{ij}g_v,u\neq v$. Thus
 \begin{equation}
 \begin{split}
 (f_{ij})_{*}(\alpha_i)&=k_1f_{ij}g_1+\dots+k_nf_{ij}g_n\\
                       &=l_1f_{ij}g_{t_1}+\dots+l_{\lambda} f_{ij}g_{t_{\lambda}}\\
                       &=c_1h_1+\dots+c_mh_m,
 \end{split}
 \end{equation}

 where $f_{ij}g_{t_1},\dots,f_{ij}g_{t_{\lambda}}$are distinct, $l_i\neq0$  and $\{t_1,\dots,t_{\lambda}\}$ is a subset of $\{1,2,\dots,n\}$.
 From above discussion, we know that $\overline{g_u}=\overline{g_v}$ if $f_{ij}g_u=f_{ij}g_v$. Then we have
 \begin{equation}
 \overline{\alpha_i}=k_1\overline{g_1}+\dots+k_n\overline{g_n}=l_1\overline{g_{t_1}}+\dots+l_{\lambda}\overline{g_{t_{\lambda}}}.
 \end{equation}
 Since $l_1f_{ij}g_{t_1}+\dots+l_{\lambda}f_{ij}g_{t_{\lambda}}=c_1h_1+\dots+c_mh_m$, $f_{ij}g_{t_s}\neq f_{ij}g_{t_w}$ if $s\neq w, h_a\neq h_b$ if $a\neq b$ $l_i\neq0,k_j\neq0$, then $\lambda=m$ and $\{f_{ij}g_{t_1},\dots,f_{ij}g_{t_{m}}\}=\{h_1\dots h_m\}$
 Then, there exists $\tau\in S_{m}$ such that $f_{ij}g_{t_{u}}=h_{\tau(u)},l_u=c_{\tau(u)}$ for all $1\leq u\leq m$.
 From the above discussion, we know $\overline{g_{t_{u}}}=\overline{h_{\tau(u)}}$ for all $1\leq u\leq m$.
 Therefore, we have
 \begin{equation}
 \begin{split}
 \overline{\alpha_i}&=k_1\overline{g_1}+\dots+k_n\overline{g_n}\\
                  &=l_1\overline{g_{t_1}}+\dots+l_{m}\overline{g_{t_{m}}}\\
                  &=c_1\overline{h_1}+\dots+c_m\overline{h_m}\\
                  &=\overline{\alpha_j}
 \end{split}
 \end{equation}

 From the above discussion, we know that for any $(\alpha_i)_{i\in\Gamma}\in lim^{\Gamma}(S^{\Gamma}_n(T,f))(n\geq0)$,
 $(p_i)_{*}(\alpha_i)=(p_j)_{*}(\alpha_j)$  if there exists an arrow from $i$ to $j$. $\Gamma$ is connected, and therefore
 $(p_i)_{*}(\alpha_i)=(p_j)_{*}(\alpha_j)\in S_n(At^{\Gamma}(T,f))$ for all $i,j\in\Gamma$.

 For each $n\geq0$, define
 \begin{equation}
 \begin{split}
  (\sigma_{(T,f)})_n:(\mathop{lim^{\Gamma}}\limits^{\bullet}(T,f))_n=lim^{\Gamma}(S^{\Gamma}_n(T,f))&\rightarrow (S(At^{\Gamma}(T,f)))_n=S_n(At^{\Gamma}(T,f))\\
  (\alpha_i)_{i\in\Gamma}&\mapsto (p_i)_{*}(\alpha_i)
  \end{split}
 \end{equation}
and we have $\partial((\sigma_{(T,f)})_n((\alpha_i)_{i\in\Gamma}))=\partial((p_i)_{*}(\alpha_i))=(p_i)_{*}(\partial(\alpha_i))=
(\sigma_{(T,f)})_{n-1}(((\partial(\alpha_i))_{i\in\Gamma})=(\sigma_{(T,f)})_{n-1}(\partial((\alpha_i)_{i\in\Gamma}))$.
Thus $ \sigma_{(T,f)}:\mathop{lim^{\Gamma}}\limits^{\bullet}(T,f)\rightarrow S(At^{\Gamma}(T,f))$ is a chain map.

For any $\theta:(T,f)\rightarrow(T^{'},f^{'})$, we will show the following diagram commutes:
  $$\xymatrix{
&\mathop{lim^{\Gamma}}\limits^{\bullet}(T,f)\ar[r]^{\sigma_{(T,f)}}\ar[d]_{\mathop{lim^{\Gamma}}\limits^{\bullet}(\theta)}&S(At^{\Gamma}(T,f))\ar[d]^{S(At^{\Gamma}(\theta))=
(At^{\Gamma}(\theta))_{*}}\\
&\mathop{lim^{\Gamma}}\limits^{\bullet}(T^{'},f^{'})\ar[r]^{\sigma_{(T^{'},f^{'})}}&S(At^{\Gamma}(T^{'},f^{'}))
}.$$
For all $n\geq0,(\alpha_i)_{i\in\Gamma}\in(\mathop{lim^{\Gamma}}\limits^{\bullet}(T,f))_n=lim^{\Gamma}(S^{\Gamma}_n(T,f))$,
\begin{equation}
\begin{split}
(\sigma_{(T,f)})_n((\alpha_i)_{i\in\Gamma})&=(p_i)_{*}(\alpha_i)\\
(At^{\Gamma}(\theta))_{*}((p_i)_{*}(\alpha_i))
&=(At^{\Gamma}(\theta)p_i)_{*}(\alpha_i)
\end{split}
\end{equation}
\begin{equation}
\begin{split}
\mathop{lim^{\Gamma}}\limits^{\bullet}(\theta)((\alpha_i)_{i\in\Gamma})&=((\theta_i)_{*}(\alpha_{i}))_{i\in\Gamma}\\
(\sigma_{(T^{'},f^{'})})_n(((\theta_i)_{*}(\alpha_{i}))_{i\in\Gamma})&=(p^{'}_{i})_{*}(((\theta_i)_{*}(\alpha_{i}))=
(p^{'}_{i}\theta_i)_{*}(\alpha_i),
\end{split}
\end{equation}
where $p^{'}_{i}:T^{'}(i)\rightarrow At^{\Gamma}(T^{'},f^{'})$, defined similarly as $p_{i}$.

It suffices to show $At^{\Gamma}(\theta)p_i=p^{'}_{i}\theta_i:T(i)\rightarrow At^{\Gamma}(T^{'},f^{'}),\forall i\in\Gamma$.
For all $x_i\in T(i)$,
\begin{equation}
At^{\Gamma}(\theta)(p_i(x_i))=At^{\Gamma}(\theta)(\overline{x_i})=\overline{\theta_i(x_i)}=p^{'}_{i}(\theta_i(x_i)).
\end{equation}
This shows the above diagram is commutative, which completes the proof.

\end{proof}

We know that $H_n(-),H_nAt^{\Gamma}(-)$ are two functors from $\textbf{Top}\mathrm{-}\textbf{Rep}\Gamma$
to $\textbf{Ab}$ for each $n\geq0$.

\begin{Corollary}\label{Cor6.11}
For any quiver $\Gamma$ with finite components, each $n\geq0$, we have a natural tranformation from $H_n(-)$ to $H_nAt^{\Gamma}(-)$.
\end{Corollary}

\end{document}